\documentclass[reqno]{amsart}
\usepackage{amsmath}
\usepackage{amsthm}
\usepackage{amssymb}
\usepackage{amscd}
\usepackage{amsfonts}
\usepackage{amsbsy}
\usepackage{mdframed}
\usepackage{enumerate,float}
\usepackage{graphicx}
\usepackage[dvips]{psfrag}
\usepackage{subfigure}
\DeclareGraphicsExtensions{.bmp,.png,.pdf,.jpg,.eps}
\usepackage{geometry}
\usepackage[utf8]{inputenc}
\usepackage[T1]{fontenc}

\usepackage[colorlinks]{hyperref}
\usepackage{xcolor}
\usepackage{tikz}
\hypersetup{
    colorlinks=true, % Enable colored links
    linkcolor=blue,   % Set the color of ToC links (e.g., blue)
    citecolor=red 
}

\graphicspath{{./figs/}} 

\DeclareGraphicsExtensions{.bmp,.png,.pdf,.jpg,.eps}

\newtheorem{teo}{Theorem}
\newtheorem{prop}{Proposition}

\newtheorem{remark}{Remark}

\newcommand{\rev}[1]{{\color{black}{{#1}}}}
\newcommand{\dd}{\textup{d}}

\newcommand{\In}{\textup{in}}
\newcommand{\out}{\textup{out}}
\newcommand{\ve}{\varepsilon}
\begin{document}

%\title[]{Singular bifurcations in a slow-fast modified Leslie-Gower model with Holling type II functional response, weak Allee effect and a generalist predator}
\title{Singular bifurcations in a modified Leslie-Gower model}

\author[]{Roberto Albarran Garc\'{\i}a}

\author[]{Martha Alvarez-Ram\'{\i}rez}
\address{Departamento de Matem\'aticas, UAM--Iztapalapa,
 09310 Iztapalapa,  Mexico City,  Mexico.}
 
\author[]{Hildeberto Jard\'on-Kojakhmetov}
\address{Johann Bernoulli Institute for Mathematics and Computer Science, University of Groningen, P.O. Box 407,
9700 AK, Groningen, The Netherlands}

\email{albarrangr74@live.com.mx, mar@xanum.uam.mx, h.jardon.kojakhmetov@rug.nl}

\subjclass[2010]{}

\begin{abstract}
We study a predator-prey system with a generalist Leslie-Gower predator, a functional Holling type II response, and a weak Allee effect on the prey. The prey's population often grows much faster than its predator, allowing us to introduce a small time scale parameter $\varepsilon$ that relates the growth rates of both species, giving rise to a slow-fast system. 
Zhu and Liu (2022) show that, in the case of the weak Allee effect, Hopf singular bifurcation, slow-fast canard cycles, relaxation oscillations, etc.  Our main contribution lies in the rigorous analysis  of a degenerate scenario organized by a (degenerate) transcritical bifurcation. 
The key tool employed is the blow-up method that desingularizes the degenerate singularity. In addition, we determine the criticality of the singular Hopf bifurcation using recent intrinsic techniques that do not require a local normal form. The theoretical analysis is complemented by a numerical bifurcation analysis, in which we numerically identify and analytically confirm the existence of a nearby Takens-Bogdanov point.
 \end{abstract}


\maketitle

\tableofcontents

\section{Introduction}\label{sec:intro}
Mathematical ecology relies heavily on the analysis of predator-prey dynamics to forecast species interactions and the evolution, growth, and population distribution patterns of the biological species involved.
The earliest predator-prey model  based on sound mathematical principles
was proposed by Lotka \cite{Lotka} and Volterra \cite{Volterra}. Following their findings, numerous attempts to improve them have been carried out since then.  The Leslie-Gower model is one of the efforts to achieve this task, whose main characteristic is that the predator's growth equation is logistic, which implies that there is competition or self-interference between the predators,  see \cite{LeslieG}.
Currently, the Leslie-Gower model  has been studied from mathematical as well as biological viewpoints by many researchers,
which is evident from the amount of published works,  see for instance, \cite{Fang2022,Gao2023,gonzalez2020,Khan2004,Qiu2023} and references therein. 

In 1931, the biologist W.C. Allee investigated population fluctuations and found out that a variety of biological phenomena affect species dynamics, such as mate-finding difficulties, social thermoregulation, genetic drift, avoidance of natural enemies, avoidance of predators, the defense of resources and poor nutrition due to low population densities. This phenomenon is now referred to as the  {\em Allee effect}.
Strong Allee effects result in critical population sizes, whereas weak Allee effects do not. 
This is due to the effect's link between population density or size to mean individual fitness.
The readers who are interested in this topic are encouraged to refer to, e.g.,  \cite{Courchamp2008,Murray2002,sch2003}.

The {\em functional response}, which refers to the amount of prey a predator consumes per unit of time, is a crucial component of mathematical models of predator-prey dynamics. The ability of the prey to flee an attack or the success of the predator's pursuit are two examples of the variables that it may affect. Numerous population studies have made the assumption that the functional response is of Holling type I, II, or III; see \cite{Holling59a,Holling59b}. 
The predator-prey model with a functional Holling type II response is the subject of this study. This type of response indicates that the predator spends more time hunting for prey at low prey densities and more time handling prey at high prey densities.
A functional Holling type II response is  expressed mathematically as:
$$ h(x)= \frac{qx}{x+a},$$
where $x=x(t)$ represents the density of the prey population at time $t\geq 0$, $q$ is the maximum per-capita consumption rate, and $a$ is the average saturation rate, that is, the number of prey items at which the predation rate reaches half its maximum value.

As we will see in the main parts of this manuscript, we will eventually assume that the prey's evolution is considerably faster than that of the predator. In this regard, Geometric singular perturbation theory (GSPT) has been developed by  Fenichel \cite{Fenichel1979}, Dumotier et al. \cite{Dumortier1996, Dumortier1992}, Krupa and Szmolyan \cite{Krupa2001exp, Krupa2001}, De Maesschalck et al. \cite{Maesschalck2021, Maesschalck}, among many others, and it is an important theory that allows one to study systems with multiple time scales.
GSPT is based on extensions of the classical normal hyperbolicity and transversality concepts. %GSPT  can be applied to study the dynamic behavior of slow-fast systems. 
We refer the reader to \cite{kuehn2015multiple} for a standard introductory treatment of the topic and to \cite{canardbook} for an overview of canard cycles for two-dimensional autonomous smooth slow-fast families of vector fields on manifolds, but especially a meticulous and unified treatment of canard cycles in dimension two.
Nowadays, GSPT, together with normally hyperbolic theory, slow-fast normal form theory, and blow-up techniques, have been used by several authors to study the dynamics of many versions of prey-predator models with multiple time scales. Without being exhaustive, we can mention the following works.
 Yao et al. \cite{Renato} discussed the cyclicity of slow-fast cycles with two canard mechanisms in the modified Holling–Tanner model.
 Wen and Shi \cite{Wen} showed the existence and uniqueness of canard cycles with cyclicity in a Leslie-Gower model are predator-prey
with prey harvesting.  Li et al. \cite{Li}  proved the phenomenon of canard explosion and relaxation oscillations in 
a Lotka-Volterra system with Allee effect in the predator. Also, the singular slow-fast homoclinic orbit phenomenon has been investigated, e.g., 
Sahoo and Guruprasad \cite{Sahoo} considered a predator-prey model with Beddington--DeAngelis functional response in which the prey reproduction is affected by the predation-induced fear and its carry-over effect. 
%\hjk{complementar con una sentencia corta en donde se explica la importancia de la homoclinica}
Using the GSPT and asymptotic expansion technique, they
observed a wide range of rich and complex dynamics such as canard cycles (with or without head) near the singular Hopf-bifurcation threshold and relaxation oscillation cycles. In \cite{Shen},  Shen analyzed the existence and the nonexistence of canard limit cycles in a predator-prey system with a non-monotonic functional response.
Yao and  Huzak \cite{YaoHuzak} studied the Leslie–Gower predator-prey model with Michaelis–Menten type prey harvesting. They focused on the cyclicity of diverse limit periodic sets, including a generic contact point, canard slow–fast cycles, transitory canards, slow–fast cycles with two canard mechanisms, singular slow–fast cycle, etc. In summary, extensive research has focused on exploring the dynamics of predator–prey models incorporating the Allee effect. 

In particular, the model we consider is taken from  Arancibia-Ibarra and Flores \cite{ibarra2021}  who studied 
the  Leslie-Gower predator-prey model  with functional response Holling type II, Allee effect in the prey, and a generalist predator, which has the form \cite{ibarra2021}:
\begin{equation}
\begin{array}{l}
\dfrac{\dd x}{\dd t} = r x \left( 1-\dfrac{x}{K}\right) (x-m)-\dfrac{qxy}{x+a},\vspace{0.2cm}\\
\dfrac{\dd y}{\dd t} = sy \left(1- \dfrac{y}{nx+c} \right),
\end{array}
\label{eq_modelo}
\end{equation}
where  $x$  and $y$ represent the size of the populations of prey and predator, respectively, at time $t$.
The biological meanings of the  parameters $r$, $K$, $m$, $q$, $a$, $s$, $n$, and $c$ are displayed in 
 Table \ref{tabla1}. 
\begin{table}[hbt]
\centering
\begin{tabular}{| c | l |}
\hline
Parameter &  Description \\\hline
$r$ & intrinsic prey growth rate \\ 
$K$ & prey environmental carrying capacity \\
$m$ &  minimum viable population or Allee threshold \\
$q$  & consuming maximum rate per capita of the predators \\
$a$ & is the amount of prey to reach half of $q$ \\
$s$ & intrinsic predator growth rate \\
$n$ & food quality and it indicates how the predators turn eaten prey\\
 & into new predator births\\
$c$ & environmental carrying capacity for the predator  \\ \hline
\end{tabular}
\caption{Biological meanings of parameters in \eqref{eq_modelo}.}\label{tabla1}
\end{table}

For its importance in the subsequent analysis, it is worth pointing out that the coordinate axes in system \eqref{eq_modelo} are invariant
for the dynamics. Due to the biological interpretation, we restrict our attention to system \eqref{eq_modelo} 
in the first quadrant $\left\{ (x,y)\in\mathbb R^2\,,\,x\geq0, \, y\geq0\right\}$.
Arancibia-Ibarra and Flores \cite{ibarra2021} demonstrated that system \eqref{eq_modelo} undergoes Hopf and Takens-Bogdanov bifurcations when certain parameter relationships are found. However, this model can also be examined from a different perspective. For instance, Zhu and Liu \cite{Zhu2022} analyzed   the same system, but assuming 
that the prey population increases much faster than that of the predators.
This leads to a slow-fast system, and since the death rate for the predator is much less than for the prey, mathematically, it results in a singular perturbation problem. With the help of slow-fast normal form theory,   geometric singular perturbation theory, and Fenichel's theory \cite{Fenichel1979},
they analyze some dynamical phenomena with strong ($m>0$) and weak ($m<0$) Allee effects, such as the existence of canard cycles related to a Hopf bifurcation, canard explosion and relaxation oscillations created by entry-exit function, and heteroclinic and homoclinic orbits. 

This article revisits the slow-fast system \eqref{eq_modelo} %introduced
also considered by Zhu and Liu \cite{Zhu2022}, focusing on a previously unaddressed degenerate case that displays intriguing dynamics characterized by the interaction of a Hopf and a degenerate transcritical singularities. This degenerate case occurs when the slow nullcline intersects the fast nullcline along the $y$-axis, as detailed in Section \ref{sec:equilibria}. In the vicinity of this case, we also uncover novel phenomena through numerical simulations that have not been previously reported. Moreover, we use modern techniques to evaluate the criticality of the organizing singularity. In contrast to the findings of Zhu and Liu \cite{Zhu2022}, our analysis identifies the conditions under which degenerate relaxation oscillations persist in a scenario organized by a transcritical (degenerate) bifurcation.

The manuscript is organized as follows.
In Section \ref{sec:equilibria}, we present the mathematical model for a modified Leslie–Gower predator–prey system incorporating an Allee effect on the prey and a Holling type II functional response. We also outline the parameter conditions required for the existence of equilibrium points, which are essential to our analysis. 
Section \ref{sec:linear_stability} contains the linear stability of these equilibrium points, highlighting its intrinsic dependence on the parameters.
Assuming that the prey population grows significantly faster than the predator population, we formulate the slow-fast model in Section \ref{sec:slowfast}, where we explore the system using geometric singular perturbation theory, for the case of a unique positive equilibrium.
Particularly, when the equilibrium point is a center, it is precisely located at the fold point. With this in mind, we then investigate the criticality of the slow-fast Hopf bifurcation using a formula from De Maesschalck et al. \cite{Maesschalck2021}, which eliminates the need for normal form transformation. This approach is explained in detail in Section \ref{sec:intrinsic}. In Section \ref{sec_desing} we desingularize the degenerate transcritical singularity by means of the blow-up technique. This allows us to later show the existence of a transitory canard and of relaxation oscillations that pass through the aforementioned singularity. To the best of our knowledge, these degenerate scenarios have not been detailed before. For completeness, the existence of the relaxation oscillation in the generic setting is presented in Section \ref{sec:relaxationoscillations}. This article ends with the discussion section.

% \hjk{The next lines may need to be rewritten, but let's wait until the paper is almost final}

% The existence of the relaxation oscillation are studied in Section \ref{sec:relaxationoscillations}.
% This article ends with the discussion section.

\section{Existence and types of positive equilibria}\label{sec:equilibria}
Before we proceed into the details of the study, for the sake of simplicity, the following scalings are introduced:
$$\begin{array}{l}
x=Ku, \qquad y=\dfrac{nK^2}{q}v, \qquad \dd\tau = \dfrac{r n K^3}{(nx +c)(u+a)}\dd t,  \qquad A= \dfrac{a}{K} < 1,\vspace{0.1cm}\\
 C=\dfrac{c}{Kn}, \qquad S=\dfrac{s}{r K},  \qquad Q=\dfrac{r K}{nq}, \qquad M=\dfrac{m}{K},\vspace{0.1cm}\\
\end{array}$$
so that $(A,M,C,S,Q)\in (0,1)\times (-1,1)\times \mathbb{R}^3_+$.
In this way, we have a system with only 5 free parameters. 

Substituting the  new variables into system \eqref{eq_modelo} and recycling $t$ to denote the rescaled time $\tau$ yields
\begin{equation}\label{eq_modelo1}
\begin{array}{l}
\dfrac{\dd u}{\dd t}  = u(u+C)\Big( (u+A)(1-u)(u-M)-v\Big),\vspace{0.2cm}\\
\dfrac{\dd v}{\dd t} = Sv(u+A)(u+C-Qv).
\end{array}
\end{equation}

%In this form, it is clear that the set $$\tilde{\Omega} =\left\{(u,v)\in \mathbb{R}^2, u\geq 0, 0 \leq v\leq \frac{1+C}{Q}.\right\}$$ is positively invariant under the flow of \eqref{eq_modelo1}. Hence, from now on $\tilde{\Omega}$ shall be the domain of definition of \eqref{eq_modelo1}.

As the change of variables is only a scale transformation, systems \eqref{eq_modelo} and \eqref{eq_modelo1}  smoothly  equivalent
and then both systems have essentially the same phase portrait, except at the singularities $(nc+c)(u+a)=0$, which are not in the biologically meaningful domain. 
%Hence, both  systems  have the same equilibrium points. 
Our aim of this manuscript is to study the model \eqref{eq_modelo} with weak Allee effect on the prey, that is
\eqref{eq_modelo1} with $M<0$.

The following result shows that the dynamics of the system, taking positive initial conditions, is contained  in a region of the first quadrant. Moreover, the solutions are bounded and do not depend on the value of $M$. This fact is biologically important because of the achievability and finiteness of populations.

\begin{prop}
\label{prop_invariante}
The set $\Gamma =\{(u,v)\in \mathbb{R}^2: 0\leq u \leq 1, v\geq 0\}$
is positively invariant under the flow of system \eqref{eq_modelo1}.  That is every orbit that enters $\Gamma$ in finite time  never leaves thereafter.
\end{prop}

\begin{proof}
It is straightforward to see that  the $u$-axis ($v = 0$) and the $v$-axis ($u = 0$) are invariant sets because \eqref{eq_modelo1}  
is of Kolmogorov type \cite{Freedman}. If $u = 1$, we have that 
$$\dfrac{\dd u}{\dd t}  =-(1+C)v < 0, \quad \text{for} \quad v>0 $$
and whatever the sign of
$$\dfrac{\dd v}{\dd t} = Sv(1+A)(1+C-Qv),$$
the trajectories  of the system get into the region  $\Gamma$ and cannot leave it once inside.
%Therefore, all solutions in the first quadrant approach, enter or stay inside the set $\Gamma$  and cannot leave it.
\end{proof}

The next proposition establishes that all solutions initiating in $\Gamma$ are eventually bounded.
\begin{prop}  All solutions of \eqref{eq_modelo1} with initial conditions in $\Gamma$ are bounded and are contained in the set 
$$\tilde{\Gamma} =\left\{(u,v): 0\leq u \leq 1, \; 0\leq v\leq\frac{1+C}{Q}\right\},$$
as $t \rightarrow \infty$.
\end{prop}
\begin{proof}
From the first equation in \eqref{eq_modelo1}, we can see that for $(u, v)\in \Gamma$ and since $C>0$, 
\begin{equation}
\label{eq3}
\dfrac{\dd u}{\dd t}\leq (1+A)(1+C)(1-u)(u-M),
\end{equation}
Applying \cite[lemma 2]{Birkhoff} to the differential inequality \eqref{eq3}, one gets that: 
$$u\leq  \dfrac{Me^{-(1+A)(1+C)(1-M)t}-e^{c_{1}}}{e^{-(1+A)(1+C)(1-M)t}-e^{c_{1}}}.$$
Then, we obtain that
$\limsup\limits_{t \to \infty} u(t)\leq 1$.

Next, since $0< u(t)< 1$  and $v> 0$, we have that
$$u+C-Qv< 1+C-Qv, \quad \text{for all} \quad v\geq 0,$$
and also $ Sv (u+C-Qv)< Sv (1+C-Qv)$. Moreover,  $u+A< 1+A$ implies that
$$ S v(u+A)(u+C-Qv)< Sv(1+A)(1+C-Qv) .$$
Thus, from the second equation of \eqref{eq_modelo1},  one has
\begin{equation}
\label{eq4}
\dfrac{\dd v}{\dd t}\leq Sv (1+A)(1+C-Qv).
\end{equation}
Again, we apply \cite[lemma 2]{Birkhoff} to \eqref{eq4}, and after integrating the differential inequality and some manipulation, 
one obtains
$$v\leq\dfrac{1+C}{e^{-S(1+A)(1+C)t+(1+C)c_{1})}+Q} .$$
It follows that,  $\limsup\limits_{t \rightarrow \infty} v(t)\leq \dfrac{1+C}{Q}$.
\end{proof}

In the next stage, we will look for all feasible equilibrium points of \eqref{eq_modelo1}. 
We note that the equilibrium points of \eqref{eq_modelo1} with $M>0$ (strong Allee effect) and where one of the species goes extinct
are $(M,0)$, $(0,0)$, $(0,C)$. On the other hand, if $M<0$ (weak Allee effect), the equilibrium point $(M,0)$ is found on the negative half-axis $u\leq 0$.
Coexistence equilibrium points,  in the presence of both populations, if any exist, correspond to the intersection of the nullclines in the first quadrant. These nullclines are
given by
\begin{equation}\label{null}
\ell(u)=\frac{u+C}{Q},  \quad  \text{and} \quad   h(u)=(u+A)(1-u)(u-M).
\end{equation}
However, instead of looking for the intersection of these curves, we analyze the positive roots of the equivalent equation
\begin{equation}\label{eq_equilibrio1}
u^3+\left(A-M-1\right)u^2 + \left(M-A-AM+\frac{1}{Q}\right)u + AM+\frac{C}{Q}=0.
\end{equation}

\begin{remark}
    The parameter $AM+\frac{C}{Q}$ is quite important. In the case $AM+\frac{C}{Q}=0$, $u=0$ is a solution of \eqref{eq_equilibrio1}. In fact, as we will show below, when $AM+\frac{C}{Q}=0$ the two nullclines $\ell(u)$ and $h(u)$ intersect at $T_{\mathcal{C}}:=(0,-AM)$ and define a degenerate scenario.
    %, further analyzed in Sections  \ref{sec_stab_deg} and \ref{sec_desing}.
\end{remark}

%We remark that  $u=0$ is clearly a solution of  \eqref{eq_equilibrio1}  when  $AM+\frac{C}{Q}=0$, which is associated with the prey-free equilibrium  $T_{\mathcal{C}}=(0,-AM)$ (boundary equilibrium). Furthermore, the intersection of  $\ell(u)$ and  $h(u)$ depends on the value of $AM+\frac{C}{Q}$.

Due to the difficulty of determining the exact solutions of equation \eqref{eq_equilibrio1}, we use  Descartes's rule of signs to give conditions under which \eqref{eq_equilibrio1} admits either three, two, one, or no positive roots. Therefore, the number of positive equilibria for system \eqref{eq_modelo1} is the same as that given in \cite[Section 4]{ibarra2021} (although our rescaling of the variables is different from the one carried out in  \cite{ibarra2021}).
%Before going any further, it should be made clear to the reader  that 
%our rescaling of the variables is different from the one carried out in  \cite{ibarra2021}.
In short, we have the following.

\begin{teo} \label{teo1}
For system \eqref{eq_modelo1}, the following statements hold.
\begin{enumerate}
\item[1.] If $AM+\frac{C}{Q}>0$ and $M<0$, (see Figure \ref{fig1}) then
\begin{enumerate}
\item[(a)] If $A-M-1<0$ and $M-A-AM+\frac{1}{Q}\leq 0$ or  $A-M-1<0$ and $M-A-AM+\frac{1}{Q}\geq 0$ or $A-M-1\geq 0$ and $M-A-AM+\frac{1}{Q}<0$, then the system \eqref{eq_modelo1} has up to two positive equilibrium points.
\item[(b)] If $A-M-1>0$ and $M-A-AM+\frac{1}{Q} \geq 0$ or $A-M-1=0$ and $M-A-AM+\frac{1}{Q}\geq 0$,
then system \eqref{eq_modelo1} has no positive equilibrium points.
\end{enumerate}

\begin{figure}[hbt]
	\subfigure[The case 1a. in Theorem \ref{teo1} with two equilibria.]{\includegraphics[width=0.3\textwidth]{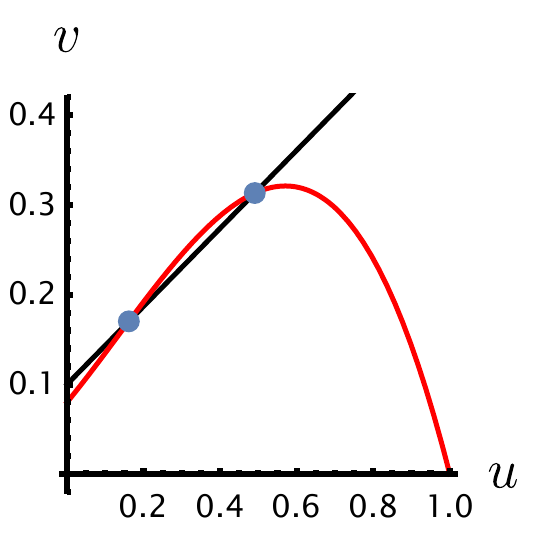}}
\hspace{0.25cm}
	\subfigure[The case 1a. in Theorem \ref{teo1} with one equilibrium.]{\includegraphics[width=0.3\textwidth]{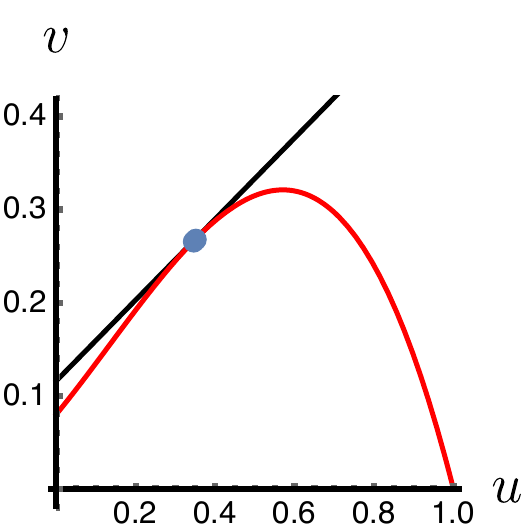}}	
 \hspace{0.25cm}
	\subfigure[The case 1b. in Theorem \ref{teo1} with no equilibria.]{\includegraphics[width=0.3\textwidth]{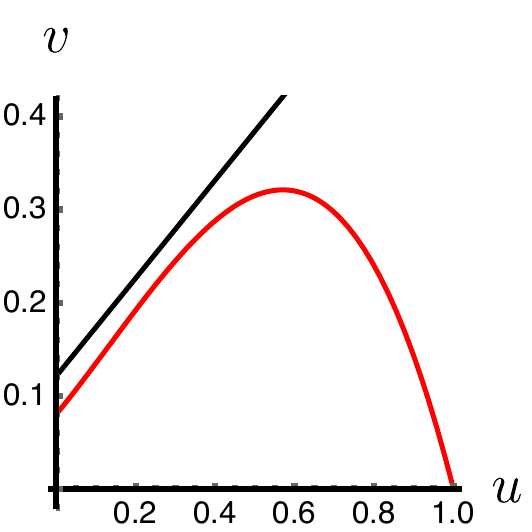}}	
\caption{The intersection of the prey nullcline $h(u)= (u+A)(1-u) (u-M)$ (red curve) and the predator nullcline
$\ell(u)=(u+C)/Q$ (black line) for \eqref{eq_modelo1} with weak ($M<0$) Allee effect.   The intersection  indicates 
two, one or zero positive  equilibrium points of system.}\label{fig1}
\end{figure}

\item[2.] If $AM+\frac{C}{Q}=0$ and $M<0$, (see Figure \ref{fig2}) then
\begin{enumerate}
\item[(a)]  If $A-M-1<0$ and $M-A-AM+\frac{1}{Q}>0$, then system \eqref{eq_modelo1}  
has up to two positive equilibrium points. 
\item[(b)] If $A-M-1<0$ and $M-A-AM+\frac{1}{Q}\leq 0$, or $A-M-1\geq 0$ and $M-A-AM+\frac{1}{Q}<0$, then system  \eqref{eq_modelo1}  has one positive equilibrium point.
\item[(c)] If $A-M-1\geq 0$ and $M-A-AM+\frac{1}{Q}\geq 0$, then system \eqref{eq_modelo1} has no positive equilibrium points.
\end{enumerate}

\begin{figure}[hbt]
	\subfigure[The case 2a. in Theorem \ref{teo1} with two positive equilibria.]{\includegraphics[width=0.3\textwidth]{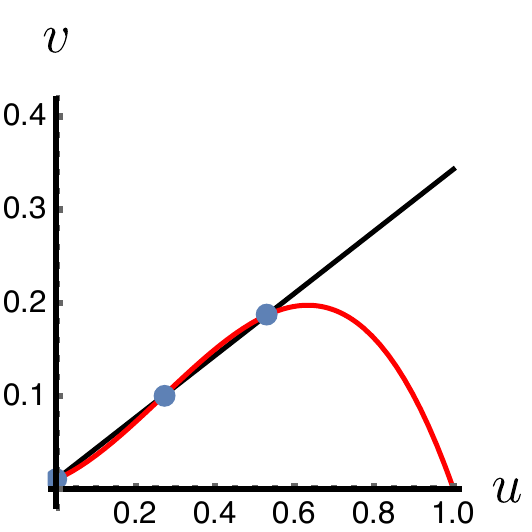}}
\hspace{0.25cm}
	\subfigure[The case 2b. in Theorem \ref{teo1} with one positive equilibrium.]{\includegraphics[width=0.3\textwidth]{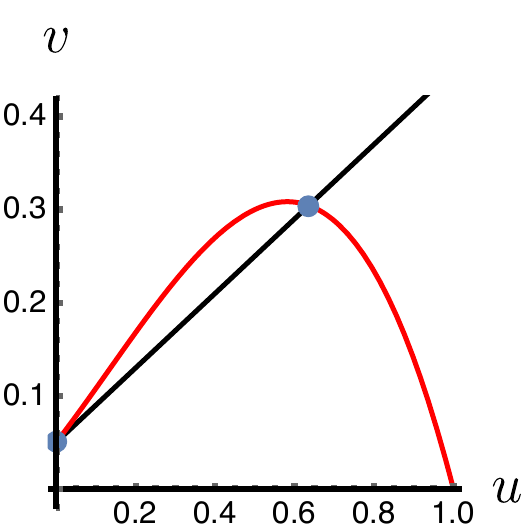}}	
 \hspace{0.25cm}
	\subfigure[The case 2c. in Theorem \ref{teo1} without positive equilibria.]{\includegraphics[width=0.3\textwidth]{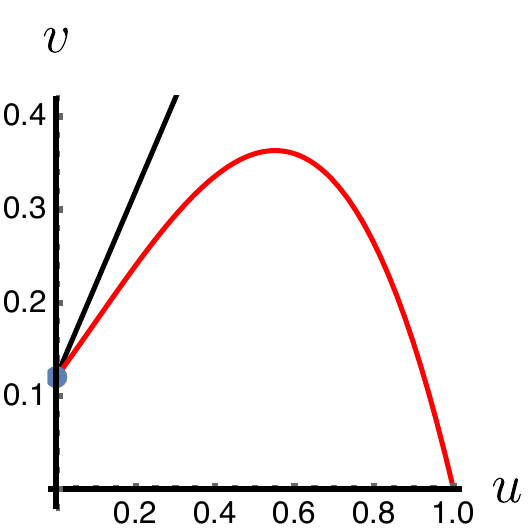}}	
\caption{The intersection of the  of the  prey nullcline $h(u)= (u+A)(1-u) (u-M)$ (red curve) and the predator nullcline
$\ell(u)=(u+C)/Q$ (black line) for \eqref{eq_modelo1} with weak ($M<0$) Allee effect.   The intersection  indicates 
two, one or zero positive  equilibrium points of system.}\label{fig2}
\end{figure}

\item[3.] If $AM+\frac{C}{Q}<0$ and $M<0$, (see Figure \ref{fig3TB}) then
\begin{enumerate}
\item[(a)]  If $A-M-1\leq 0$ and $M-A-AM+\frac{1}{Q}\leq 0$ or $A-M-1\geq 0$  and $M-A-AM+\frac{1}{Q} \geq 0$ or $A-M-1\geq 0$ and $M-A-AM+\frac{1}{Q}\leq 0$,
then system  \eqref{eq_modelo1}  has one positive equilibrium point.
\item[(b)] If $A-M-1 < 0$ and $M-A-AM+\frac{1}{Q}>0$, then system \eqref{eq_modelo1} has up to three positive equilibrium points.

\begin{figure}[hbt]
	\subfigure[The case 3a. in Theorem \ref{teo1} with one positive equilibrium.]{\includegraphics[width=0.3\textwidth]{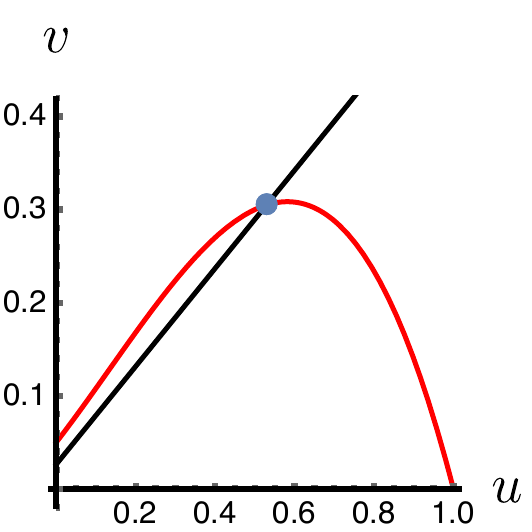}}
\hspace{0.25cm}
\subfigure[The case 3b. in Theorem \ref{teo1} with two positive equilibria.]{\includegraphics[width=0.3\textwidth]{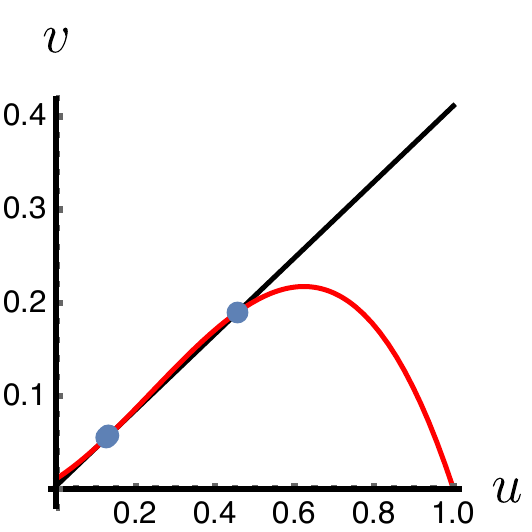}}
\hspace{0.25cm}
	\subfigure[The case 3b. in Theorem \ref{teo1} with three positive equilibria.]{\includegraphics[width=0.3\textwidth]{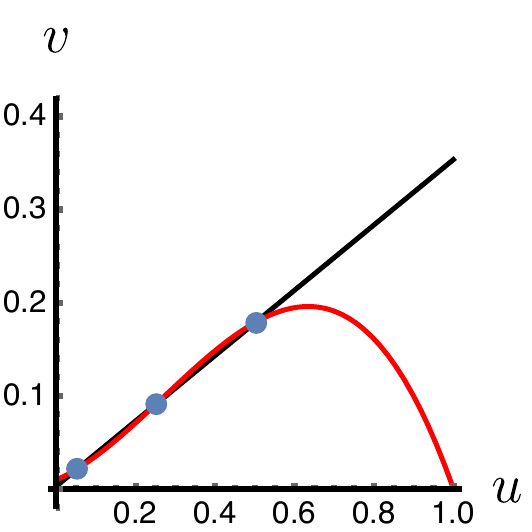}}	
	\hspace{0.1cm}	
\caption{The intersection of the  of the  prey nullcline $h(u)= (u+A)(1-u) (u-M)$ (red curve) and the predator nullcline
$\ell(u)=(u+C)/Q$ (black line) in the \eqref{eq_modelo1} model with weak ($M<0$) Allee effect.   The intersection  indicates one or up to three positive  equilibrium points of the system.}\label{fig3TB}
\end{figure}
\end{enumerate}
\end{enumerate}
\end{teo}

\begin{remark}$ $
\label{remark2}
\begin{itemize}
\item Case $2$ in Theorem \ref{teo1} corresponds to a degenerate case where the two nullclines intersect along the $v$-axis, see Figure \ref{fig2}. It also serves as a separatrix of generic situations already studied in the literature, see for instance \cite{zhong}. Therefore, an important part of our paper is dedicated to such a case.
    \item Consider case $2$ in Theorem \ref{teo1} with $A$ and $M$ fixed. Let $Q$ be sufficiently small such that there is exactly one positive equilibrium point (to be denoted as $E_1=(\mathcal{U}_1,\mathcal{V}_1)$ in the rest of the paper). If $A-M-1<0$ then, as one increases the value of $Q$, the coordinate $\mathcal{U}_1$ decreases. Eventually, for $Q=\frac{1}{A-M+AM}$, a second equilibrium point (to be denoted $E_2=(\mathcal{U}_2,\mathcal{V}_2)$ in the rest of the paper) appears ($2(b)\rightarrow 2(a)$). Further increasing $Q$ makes $E_1$ collide with $E_2$ and disappear ($2(a)\rightarrow 2(c)$). On the other hand, if $A-M-1>0$ then, as one increases the value of $Q$,  the coordinate $\mathcal{U}_1$ decreases.  Eventually, $E_1$ collides with the origin before ceasing to exist ($2(b)\rightarrow 2(c)$).
    \item Consider case $3$ in Theorem \ref{teo1} with $A$ and $M$ fixed. Let $Q$ be small enough that there is exactly one positive equilibrium point $E_1=(\mathcal{U}_1,\mathcal{V}_1)$. If $A-M-1<0$, then as the value of $Q$ is increased, the coordinate $\mathcal{U}_1$ decreases. Moreover, for $Q=\frac{1}{A-M+AM}$ we have two cases: 
\begin{enumerate}[(i)]
\item If $Q>\frac{3}{1+A+A^{2}-M+AM+M^{2}}$ a second equilibrium point $E_2=(\mathcal{U}_2,\mathcal{V}_2)$ appears and eventually a third equilibrium point $E_{3}$ ($3(a) \rightarrow 3(b)$). When $Q$ grows even more the points $E_1$ and $E_2$ collide and eventually disappear, while $E_{3}$ collides with the origin before ceasing to exist.
\item If $Q<\frac{3}{1+A+A^{2}-M+AM+M^{2}}$, there is always only one  equilibrium point.
\end{enumerate}

\vspace*{0.25cm}

 On the other hand, if $A-M-1>0$, we always have a unique equilibrium point.  This because we are moving from one condition to another within those considered in item 3(a) of Theorem \ref{teo1}. 
\end{itemize}
\end{remark}

%\hjk{Entonces, nos restringimos a $A-M-1>0$?, o estudiamos todo? Este es un punto importante y define el resto del paper} \textcolor{green}{Estamos considerando $A-M-1>0$ y $A-M-1<0$, pues en ambos casos hay un único punto de equilibrio como lo describimos en 3(a) del teorema 1}

\begin{remark}
    In the rest of the paper, we restrict ourselves to the case where there is only one equilibrium point for $C=-AMQ$ and $C<-AMQ$, but see section \ref{sec:numerics} for a further discussion. The case $C>-AMQ$ does not present oscillations, which are our main interest. Moreover, we notice that $0<Q<\frac{1}{A-M+AM}$ is necessary to have a unique positive equilibrium, which, in turn, implies that $A-M+AM>0$.
\end{remark}

\section{Linear stability analysis of the positive equilibrium \texorpdfstring{$E_{1}$}{E1}}\label{sec:linear_stability}

Since the $ u$ coordinate of the equilibrium point obtained from \eqref{eq_equilibrio1} has a very lengthy expression, namely
\begin{equation}
\label{eq_U1}
\begin{split}
\mathcal{U}_{1}&=\frac{1}{6} \left(2^{2/3}\sqrt[3]{\alpha}+\frac{2 \sqrt[3]{2} \Big(Q \big(A^2+(A-1) M+A+M^2+1\big)-3\Big)}{Q \sqrt[3]{\alpha}}+2 (-A+M+1)\right) 
\end{split}
\end{equation}
where
\begin{equation*}
\begin{split}
\alpha&=-2 A^3-3 A^2 M-3 A^2-\frac{9 (-A+3 C+M+1)}{Q}+3 A M^2-12 A M+3 A+2 M^3-3 M^2\\
&-3 M+2+\Bigg[ \frac{1}{Q^3}\Bigg(Q \Big(2 A^3 Q+3 A^2 (M+1) Q-3 A ((M-4) M Q+Q+3)+27 C\\
&-(M+1) ((M-2) (2 M-1) Q-9)\Big)^2-4 \Big(Q \big(A^2+(A-1) M+A+M^2+1\big)-3\Big)^3\Bigg) \Bigg]^{1/2}
\end{split}
\end{equation*}
we will simplify the presentation by using $\mathcal U_1$ as a parameter. The reader should keep in mind, however, that this is just a convenience, and $\mathcal U_1$ is entirely determined by the parameters of \eqref{eq_modelo1}.

In order to discuss the stability of the positive equilibrium point, we need to compute the Jacobian matrix 
of system \eqref{eq_modelo1} with weak Allee effect ($M<0$) at $E_{1}=(\mathcal{U}_{1},\mathcal{V}_{1})$ with $0<\mathcal{U}_{1}<1$ and $\mathcal{V}_{1}=\dfrac{\mathcal{U}_{1}+C}{Q}$.  The resulting matrix is
% \begin{equation}\label{matriz1}
% \mathcal{J}(E_{1})=\: \begin{pmatrix}
% \mathcal{U}_{1}(\mathcal{U}_{1}+C)\Big(A-M+AM+2(M+1-A)\mathcal{U}_{1}-3\mathcal{U}_{1}^{2}\Big) & -\mathcal{U}_{1}(\mathcal{U}_{1}+C)\\
% \dfrac{S(\mathcal{U}_{1}+A)(\mathcal{U}_{1}+C)}{Q} & -S (\mathcal{U}_{1}+A)(\mathcal{U}_{1}+C)
% \end{pmatrix}.
% \end{equation}
\begin{equation}\label{matriz1}
\mathcal{J}= \begin{pmatrix}
\mathcal{U}_{1}(\mathcal{U}_{1}+C)\Big(A-M+AM+2(M+1-A)\mathcal{U}_{1}-3\mathcal{U}_{1}^{2}\Big) & -\mathcal{U}_{1}(\mathcal{U}_{1}+C)\\
\dfrac{S(\mathcal{U}_{1}+A)(\mathcal{U}_{1}+C)}{Q} & -S (\mathcal{U}_{1}+A)(\mathcal{U}_{1}+C)
\end{pmatrix}.
\end{equation}
Thus, the  trace  and the determinant of   $\mathcal{J}$ are given by
\begin{equation}\label{traza1}
 \text{tr}(\mathcal{J})=\: (\mathcal{U}_{1}+C)\Big((A-M+AM)\mathcal{U}_{1}+2(M+1-A)\mathcal{U}_{1}^{2}-3\mathcal{U}_{1}^{3}-S (\mathcal{U}_{1}+A)\Big),
 \end{equation}
and
\begin{equation}\label{det1}
 \det(\mathcal{J})=\: S \mathcal{U}_{1}(\mathcal{U}_{1}+A)(\mathcal{U}_{1}+C)^{2}\Big( \frac{1}{Q}-(A-M+AM)\mathcal{U}_{1}+2(M+1-A)\mathcal{U}_{1}^{2}-3\mathcal{U}_{1}^{3}\Big),
 \end{equation}
respectively. It is clear that the signs of the eigenvalues are determined by $ \text{tr}(\mathcal{J})$ and $\det(\mathcal{J})$.
%Thus, we obtain the theorem below, where the reader should be warned that unlike what happens in \cite{ibarra2021}, in this work, and similar to \cite{Zhu2022}, we will use $Q$ as the bifurcation parameter rather than $S$, since $Q$ allows us to give conditions on the position and stability of the point, while $S$ only gives us conditions on the stability..

%%%%%%%%%%%%%%%% OJO

\subsection{Linear stability analysis for the case \texorpdfstring{$C<-AMQ$}{C<-AMQ} }\label{sec41}

In the following theorem, we determine the local stability properties of the unique positive equilibrium for the case $C<-AMQ$.  The reader should be warned that unlike what happens in \cite{ibarra2021}, in this work,
and similar to \cite{Zhu2022}, we will use $Q$ as the bifurcation parameter rather than $S$ since $Q$ allows us to give conditions on the position and stability of the point, while $S$ only gives us conditions on the stability.

%\hjk{I think that the next theorem is inconsistent with the Takens-Bogdanov argument, or maybe it is incomplete. I imagine there should be another value of $Q$ for which $E_1$ is nilpotent then} \textcolor{green}{Efectivamente la simulación numérica nos muestra que la TB existe para múltiples puntos de equilibrio. Cuando decidimos considerar esta opción es porque estábamos en la busqueda de justificar la existencia de una Hopf subcrítica, pero ahora nos representa problema. Proponemos quitarla y ya ir cerrando el artículo.}

\begin{teo}
\label{teoest}
Let the system \eqref{eq_modelo1}  be such that 
we are in case 3(a) of Theorem \ref{teo1} \footnote{Where 
 $M<0$, $C<-AMQ$, $A-M-1< 0$ and $M-A-AM+\frac{1}{Q}< 0$ or $A-M-1> 0$  and $M-A-AM+\frac{1}{Q} > 0$ or $A-M-1> 0$ and $M-A-AM+\frac{1}{Q}< 0$.}.  Then, the system \eqref{eq_modelo1} has only one positive interior equilibrium point $E_{1}=(\mathcal{U}_{1},\mathcal{V}_{1})$, which is:
\begin{enumerate}
\item[1.] a stable node if $3\mathcal{U}_{1}^{2}-2(M+1-A)\mathcal{U}_{1}-(AM+A-M-\frac{1}{Q})>0$ and
$$Q> \frac{3(\mathcal{U}_{1}+C)}{(S+M(\mathcal{U}_{1}-2)-\mathcal{U}_{1})\mathcal{U}_{1}+A(S-(\mathcal{U}_{1}-2)\mathcal{U}_{1}+M(2\mathcal{U}_{1}-3)};$$
\item[2.]  an unstable node if $3\mathcal{U}_{1}^{2}-2(M+1-A)\mathcal{U}_{1}-(AM+A-M-\frac{1}{Q})>0$ and
$$Q< \frac{3(\mathcal{U}_{1}+C)}{(S+M(\mathcal{U}_{1}-2)-\mathcal{U}_{1})\mathcal{U}_{1}+A(S-(\mathcal{U}_{1}-2)\mathcal{U}_{1}+M(2\mathcal{U}_{1}-3)};$$
\item[3.]  a linear center if $3\mathcal{U}_{1}^{2}-2(M+1-A)\mathcal{U}_{1}-(AM+A-M-\frac{1}{Q})>0$ and
$$Q= \frac{3(\mathcal{U}_{1}+C)}{(S+M(\mathcal{U}_{1}-2)-\mathcal{U}_{1})\mathcal{U}_{1}+A(S-(\mathcal{U}_{1}-2)\mathcal{U}_{1}+M(2\mathcal{U}_{1}-3)}.$$ 
\end{enumerate}
\end{teo}

\begin{remark}
    Due to the importance of the critical value of the parameter $Q$ corresponding to the linear center, we shall denote
    \begin{equation}
        Q_H:=\frac{3(\mathcal{U}_{1}+C)}{(S+M(\mathcal{U}_{1}-2)-\mathcal{U}_{1})\mathcal{U}_{1}+A(S-(\mathcal{U}_{1}-2)\mathcal{U}_{1}+M(2\mathcal{U}_{1}-3)},
    \end{equation}
    which, as we will see below, corresponds to the parameter value for which a singular Hopf bifurcation occurs.
\end{remark}

\begin{proof}
In order to determine the stability,  we must calculate the eigenvalues of the Jacobian matrix  \eqref{matriz1}. 
 We first note that  \eqref{traza1} and  \eqref{det1}  may be written as
\begin{equation}
    \text{tr}(\mathcal{J})=\: (\mathcal{U}_{1}+C)\Big(\mathcal{U}_{1}J_{11}-S (\mathcal{U}_{1}+A)\Big),
\end{equation}
and
\begin{equation}
    \det(\mathcal{J})=\: S \mathcal{U}_{1}(\mathcal{U}_{1}+A)(\mathcal{U}_{1}+C)^{2}\Big( \frac{1}{Q}- J_{11}\Big),
\end{equation}
respectively,
% 
% $$ \text{tr}(\mathcal{J})=\: (\mathcal{U}_{1}+C)\Big(\mathcal{U}_{1}J_{11}-S (\mathcal{U}_{1}+A)\Big),\qquad \mbox{and}\quad 
%  \det(\mathcal{J}(E_{1}))=\: S \mathcal{U}_{1}(\mathcal{U}_{1}+A)(\mathcal{U}_{1}+C)^{2}\Big( \frac{1}{Q}- J_{11}\Big),$$
with $J_{11}=A-M+AM+2(1-A+M)\mathcal{U}_{1}-3\mathcal{U}_{1}^{2}$. Then, the signs of $\text{tr}(\mathcal{J})$ and $\det(\mathcal{J})$ depend on the signs of the factors $\mathcal{U}_{1}J_{11}-S (\mathcal{U}_{1}+A)$ and $\frac{1}{Q}- J_{11}$, respectively.
Therefore, if  $\frac{1}{Q}- J_{11}<0$, then  $\text{det}(\mathcal{J})<0$ and thus the equilibrium point  is a saddle point. 

Next, if
$\frac{1}{Q}- J_{11}>0$,  it follows that   $\det(\mathcal{J})>0$. Then, the behavior of the equilibrium point depends on the trace.
For convenience, we  do some algebraic manipulations to arrive at the fact that the expression \eqref{traza1}  reduces to 
% \begin{equation}
% \label{traza}
% \footnotesize{
% \text{tr}(\mathcal{J})=(\mathcal{U}_{1}+C)\left(3\left(AM+\frac{1}{Q}\right)+\mathcal{U}_{1} \left( (A-M-1)\mathcal{U}_{1}+2\left(\frac{1}{Q}-A+M-AM \right)+\frac{1}{Q}\right)-S (\mathcal{U}_{1}+A)\right).}   
% \end{equation}
%
\begin{equation}\label{traza}
\begin{split}
    \text{tr}(\mathcal{J})&=(\mathcal{U}_{1}+C)\left[3\left(AM+\frac{1}{Q}\right)\right.\\
    &\quad\left.+\mathcal{U}_{1} \left( (A-M-1)\mathcal{U}_{1}+2\left(\frac{1}{Q}-A+M-AM \right)+\frac{1}{Q}\right)-S (\mathcal{U}_{1}+A)\right]. 
\end{split}
\end{equation}

We refer to Appendix \ref{app:JacobianTrace} for the derivation of \eqref{traza}.
Furthermore, after some algebraic manipulation, one obtains that
$\text{tr}(\mathcal{J})=0$ at the  critical value for case 3(a), i.e. with $C<-AMQ$, namely
$$Q_H= \frac{3(\mathcal{U}_{1}+C)}{(S+M(\mathcal{U}_{1}-2)-\mathcal{U}_{1})\mathcal{U}_{1}+A(S-(\mathcal{U}_{1}-2)\mathcal{U}_{1}+M(2\mathcal{U}_{1}-3)}.$$
If we assume that $Q=Q_H$ with $\det(\mathcal{J})>0$,  we  conclude that  the matrix $\mathcal{J}$  
has a pair of conjugate, purely imaginary eigenvalues, so that the equilibrium $E_1$ is a centre.  Whenever $Q>Q_H$ the (complex) eigenvalues 
of $\mathcal{J}$  have negative real parts negative, thus $E_1$ is topologically a stable node, while  $Q<Q_H$  implies that $\mathcal{J}$
has eigenvalues with positive real parts such that the equilibrium point is, topologically, an unstable node. 
Considering these different possibilities for the two eigenvalues,  we obtain the results.
\end{proof}

In summary,  from the above analysis assuming that $\frac{1}{Q}-J_{11}>0$, we can say that  $E_1$ is stable for $Q>Q_{H}$, and it is unstable for $Q<Q_{H}$. In addition, the equilibrium point $E_1$  loses its stability through a supercritical Hopf bifurcation at the bifurcation value $Q=Q_H$, which we further describe in Section \ref{sec:intrinsic}, so $E_1$ is surrounded by a stable limit cycle. For more information on this subject, the reader is referred to \cite{Guckenheimer}.

We would like to highlight that we are focusing exclusively on the case where $\frac{1}{Q}-J_{11}>0$, as oscillations naturally occur in this scenario, whereas the case 
$\frac{1}{Q}-J_{11}<0$ is more complex.

\subsection{Linear stability analysis for the \emph{degenerate} case \texorpdfstring{$C=-AMQ$}{C=-AMQ}}
The following result is similar to Theorem \ref{teoest}, but for the degenerate case $C=-AMQ$.
% , with a critical value.
% \begin{equation}
% \label{QH}
%     Q_{H}=  \frac{3\mathcal{U}_{1}}{AS+(2A-2M+2AM+S)\mathcal{U}_{1}+(M-A-1)\mathcal{U}_{1}^2}
% \end{equation}
\begin{teo}
\label{teoest1}
Let the system \eqref{eq_modelo1}  be such that,
we are in case 2(b) of Theorem \ref{teo1} \footnote{Where 
 $M<0$, $C=-AMQ$, $A-M-1>0$ and $M-A-AM+\frac{1}{Q} <0$ or $A-M-1<0$ and $M-A-AM+\frac{1}{Q} <0$.}.  Then, the system \eqref{eq_modelo1} has only one positive interior equilibrium point. Therefore, the  equilibrium point  $(\mathcal{U}_{1},\frac{\mathcal{U}_{1}+C}{Q})$ is:
\begin{enumerate}
\item[1.] a stable node if $3\mathcal{U}_{1}^{2}-2(M+1-A)\mathcal{U}_{1}-(AM+A-M-\frac{1}{Q})>0$ and
$$Q> \frac{3\mathcal{U}_{1}}{AS+(2A-2M+2AM+S)\mathcal{U}_{1}+(M-A-1)\mathcal{U}_{1}^2};$$
\item[2.]  an unstable node if $3\mathcal{U}_{1}^{2}-2(M+1-A)\mathcal{U}_{1}-(AM+A-M-\frac{1}{Q})>0$  and
$$Q<  \frac{3\mathcal{U}_{1}}{AS+(2A-2M+2AM+S)\mathcal{U}_{1}+(M-A-1)\mathcal{U}_{1}^2};$$
\item[3.]  a linear center if $3\mathcal{U}_{1}^{2}-2(M+1-A)\mathcal{U}_{1}-(AM+A-M-\frac{1}{Q})>0$  and
\begin{equation}
\label{QH}
    Q=Q_{H}=  \frac{3\mathcal{U}_{1}}{AS+(2A-2M+2AM+S)\mathcal{U}_{1}+(M-A-1)\mathcal{U}_{1}^2}.
\end{equation}
%$$Q= Q_H= \frac{3\mathcal{U}_{1}}{AS+(2A-2M+2AM+S)\mathcal{U}_{1}+(M-A-1)\mathcal{U}_{1}^2};$$
%\item [4.] a saddle if $3w^{2}-2(M+1-A)w-(AM+A-M-\frac{1}{Q})<0$
\end{enumerate}
\end{teo}

\begin{proof} The proof follows from that of Theorem \ref{teoest} by substituting $C=-AMQ$.
% Analogous to the Theorem \ref{teoest}, we obtain
% $\text{tr}(\mathcal{J})=0$ at the  critical value for case 2(b),  with $C=-AMQ$
% $$Q_H= \frac{3\mathcal{U}_{1}}{AS+(2A-2M+2AM+S)\mathcal{U}_{1}+(M-A-1)\mathcal{U}_{1}^2}.$$
\end{proof}

 The results above show that, within the case of a unique positive equilibrium, the unique equilibrium $E_1$ undergoes a Hopf bifurcation as the parameter $Q$ is varied. We want to describe how this bifurcation leads to complex dynamics in the model. For this, we will exploit a slow-fast structure provided by the parameter $S$. This is developed in the following section.

\section{Analysis of the slow-fast model}\label{sec:slowfast}

Species that belong to different trophic levels have distinct
 growth rates, which may differ by several orders of magnitude.
In particular, there is empirical evidence of the interaction between species such as insects and birds \cite{Ludwig},  lynx and hares \cite{Stenset}, etc., that show that the growth rate of the prey is often much faster than that of its predator. 
 Based on these observations, several researchers have considered a mathematical prey-predator model accounting for the difference in time-scale.
More specifically, one introduces a small time-scale parameter $\varepsilon$,  $0 < \varepsilon \ll 1$  in the basic model.
 The parameter $\varepsilon$ is interpreted as the ratio between the linear death rate of the predator and the linear growth rate of the prey \cite{Hek}, and the assumption $\varepsilon< 1$ implies that one generation of predator can encounter many different generations of prey.

In this section, we study the dynamics of \eqref{eq_modelo1} under the assumption that the prey population grows much faster than the predator's so that $s\ll r$.  Since $K$, the carrying capacity for the prey population,
usually has a value greater than $1$, then  $S = \frac{s}{rK}$ is a perturbation parameter that is small enough. 
Consequently, system \eqref{eq_modelo1} can be regarded as a slow-fast system \cite{kuehn2015multiple} with a fast time scale $t$, where $u$ is a fast variable, and $v$ is a slow variable. The slow-fast model is given by the following system of equations: 
\begin{equation}
\label{slowfast1}
\begin{array}{l}
\dfrac{\dd u}{\dd t}= f(u,v)=\: u(u+C)\big((u+A)(1-u)(u-M)-v \big), \vspace{0.2cm}\\
\dfrac{\dd v}{\dd t}= \varepsilon g(u,v)=\: \varepsilon v (u+A)(u-Qv+C),
\end{array}
\end{equation}
 where $\varepsilon > 0$ is a small parameter.

In the limiting case $\varepsilon \to 0$, the system  \eqref{slowfast1} yields the {\em layer} (or {\em fast}) {\em problem}:
\begin{equation}
\label{slowfast2}
\begin{split}
\frac{\dd u}{\dd t}&= \: u(u+C)\big((u+A)(1-u)(u-M)-Qv \big), \vspace{0.2cm}\\
\frac{\dd v}{\dd t}&= 0,
\end{split}
\end{equation}
where the slow variable $v$  can be treated as a parameter in the first equation. Next, by re-writing system \eqref{slowfast1}  on the slow time scale $\tau = \varepsilon t$, we obtain 
\begin{equation}
\label{slowfast3}
\begin{split}
\varepsilon\dfrac{\dd u}{\dd\tau}&= f(u,v)=\: u(u+C)\big((u+A)(1-u)(u-M)-v \big), \vspace{0.2cm}\\
\dfrac{\dd v}{\dd\tau}&= g(u,v)=\: v (u+A)(u-Qv+C).
\end{split}
\end{equation}
Taking the singular limit as $\varepsilon \to 0$ in system \eqref{slowfast3} leads to the {\em reduced}   (or {\em slow}) {\em problem}:
\begin{equation}
\label{slowfast4}
\begin{split}
0&= f(u,v)=\: u(u+C)\big((u+A)(1-u)(u-M)-v \big), \vspace{0.2cm}\\
\dfrac{\dd v}{\dd\tau}&= g(u,v)=\:  v (u+A)(u-Qv+C),
\end{split}
\end{equation}
which is a differential-algebraic system.

The equilibria of the layer subsystem \eqref{slowfast3} define the  critical manifold
$$\mathcal{M}_0=\:\{ (u,v)\in \mathbb{R}_{+}^{2} \mid  u(u+C)\big((u+A)(1-u)(u-M)-v \big)=0\},$$
composed of two submanifolds, namely
\begin{equation*}
\begin{array}{l}
\mathcal{M}_{0}^{0}=\: \{ (u,v)\in \mathbb{R}_{+}^{2} \mid u=0, v\geq 0 \},\vspace{0.2cm}\\
\mathcal{M}_{0}^{1}=\: \{ (u,v)\in \mathbb{R}_{+}^{2}  \mid v= (u+A)(1-u)(u-M) =h(u), \;  \text{with} \;  0<u<1,  v>0 \},
\end{array}
\end{equation*}
such that $\mathcal{M}_0 = \mathcal{M}_0^0  \cup \mathcal{M}_0^1$. Indeed,  $\mathcal{M}_{0}^{0}$ is the positive $v$-axis and  the  $\mathcal{M}_{0}^{1}$  is cubic-shaped.  
Notice that  $\mathcal{M}_0$ is the phase space of the reduced system \eqref{slowfast4}.

Since $0\leq u \leq 1$, we note that $\mathcal{M}_0^1$ has a unique zero at $u=1$  and  its intersection with the  positive $v$-axis is the point  $T_{\mathcal C}=(0,v_{\mathcal C})$, with  $v_{\mathcal C}=-AM >0$, since $M<0$. We remark  that the shape of $\mathcal{M}_0$  plays a key role in determining the behavior of system \eqref{slowfast1}.  Now, from the derivative 
\begin{equation}\label{eq:dh}
    \frac{\dd h}{\dd u} = -3u^2+2(1-A+M) u + (A-M+AM),
\end{equation}
 it can be verified that the function $h(u)$ has a unique maximum point at $P:=(u_p,v_p)$ in $\mathcal{M}_0^1$, where
\begin{equation} \label{fold_point}
\begin{array}{l}
u_p= \frac{1}{3}(1-A+M+\beta ),\vspace{0.2cm}\\
v_p = \frac{1}{27}\big(1-A-2M+\beta\big) \big(2+A-M-\beta\big) \big(1+2A+M+\beta\big),
\end{array}
\end{equation}
with $\beta= \sqrt{A^2+AM+A+M^{2}-M+1}$. The shape of the critical curve 
%and the slow-fast dynamic, where  ${P}:=(u_p,v_p)$ denotes the fold point, 
is illustrated in Figure \ref{dinamicasf}.

Next, one can verify from \eqref{eq:dh}, that
\begin{equation}
    \frac{\dd h}{\dd u}\begin{cases}
        >0, & 0<u<u_p,\\
        <0, & u_p<u<1.
    \end{cases}
\end{equation}

Therefore, we define the attracting and repelling branches of $\mathcal M_0^1$ as:
\begin{equation}\label{m01a}
\begin{split}
\mathcal{M}_{0}^{1,r}&= \{ (u,v)\in\mathcal M_0^1 \; | \;  0\leq u<u_p \},\\
\mathcal{M}_{0}^{1,a}&= \{ (u,v)\in\mathcal M_0^1 \; |\;   u_p<u\leq 1 \}, 
\end{split}
\end{equation}
see again Figure \ref{dinamicasf}.

%  \hjk{which should this Figure be?}. Furthermore, the dynamics of the fast subsystem \eqref{slowfast2} and the slow subsystem \eqref{slowfast3} are
%illustrated in Figure \ref{dinamicasf}. The fold point shall be denoted by ${P}=(u_p,v_p)$.

\begin{remark}
    Notice that, for the case of unique positive equilibrium, its stability has already been characterized in Theorems \ref{teoest} and \ref{teoest1}. In particular, the case when the equilibrium point is a center occurs when the equilibrium point is located exactly at the fold point ${P}$. 
\end{remark}

\begin{figure}[hbt]
	\subfigure[$C< -AMQ$]{\includegraphics[width=0.4\textwidth]{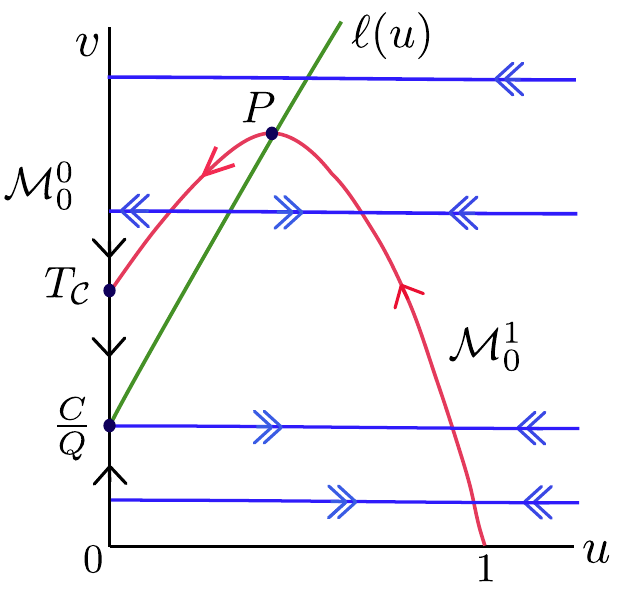}}
\hspace{1.5cm}
\subfigure[$C= -AMQ$]{\includegraphics[width=0.4\textwidth]{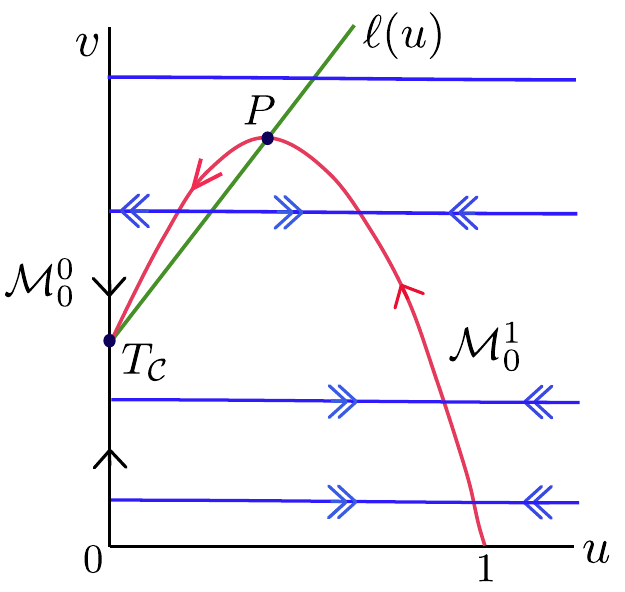}}
\caption{Dynamics of the slow-fast system \eqref{slowfast1}  for 
(a) $C< -AMQ$ and (b) $C= -AMQ$, where 
 $\mathcal{M}_0^0$ and $\mathcal{M}_0^1$ are the critical submanifolds. The fold point $P=(u_p,v_p)$ and $T_\mathcal{C}=(0,-AM)$ are shown, as well as a schematic representation of the slow flow in the different components of the critical manifold. The green line $\ell (u)$ denotes the nullcline of variable the $v$. Hence, we are depicting the case in which an equilibrium of the slow flow coincides with the fold point, which may lead to canards, and corresponds to the parameter $Q=Q_H$. Indeed, in the figures, the magenta curve is a singular canard orbit.}
 %collector $\mathcal{M}_0$ with $M<0$.
%The branch from $\mathcal{M}_0^1$ from $T_\mathcal{C}$ to ${P}$ is repellent and its branch from ${P}$ to $(1,0)$ is attractive.
%The manifold $\mathcal{M}_0^0$ is along the positive $v$ axis. Attract by $\overline{v_p T_{\mathcal C}}$ and repel by $\overline{T_{\mathcal C} 0}$. 
%
%${P}$ is the fold point and $T_{\mathcal C}$ is the transcritical branch point for the fast subsystem. 
%The horizontal flow lines correspond to the layer equation flow (fast flow). 
%\textcolor{blue}{This image represents the cases $C<-AMQ$ and $C=-AMQ$}. \hjk{I don't think this is correct. For $C=-AMQ$ the flow along $v$ approaches $T_C$ on both sides}.}
%%\label{parabola}
\label{dinamicasf}
\end{figure}

Regarding the flow in the critical curve $\mathcal{M}_{0}^{0}$, taking into account that the line $u=0$ is invariant according to the Proposition \ref{prop_invariante}, it is possible to establish that
\begin{equation}
    \frac{\dd v}{\dd\tau}\bigg|_{u=0}=Av(c-Qv) \left\{ \begin{array}{ll}  > 0, & \text{if} \;\; 0 < v < \frac{C}{Q},\\ 
< 0, & \text{if} \;\;  \frac{C}{Q}<v. \end{array}\right.
\end{equation}

Next, we study the slow reduced dynamics on the critical manifold $\mathcal{M}_{0}^{1}$.
The slow dynamics along $\mathcal{M}_{0}^{1}$ is given by the following equation
\begin{equation}
\label{eq_flujo}
\frac{\dd u}{\dd \tau}\bigg|_{\mathcal M_0^1}=\frac{g(u,h(u))}{h'(u)},  
\end{equation} 
where $g(u,v)=\dfrac{\dd v}{\dd\tau}$ and $h(u)$ are given in \eqref{null}. So,  it is possible to establish that
\begin{equation}\label{eq:flujo2}
\frac{\dd u}{\dd \tau}\bigg|_{\mathcal M_0^1}=\frac{g(u,h(u))}{h'(u)}=\frac{h(u)(u+A)(u+C-Qh(u))}{h'(u)}.
\end{equation}
The sign of \eqref{eq_flujo} depends on the sign of $g(u,h(u))$ and $h'(u)$. Notice that $h'(u)=0$ at the fold point $P$. So, we proceed with the desingularization \cite{Takens76} by multiplying \eqref{eq:flujo2} by $h'(u)$ (we can also divide out $h(u)(u+A)$) to get the desingularized equation
\begin{equation}\label{eq:flujo3}
    \frac{\dd u}{\dd \tau} = u+C-Qh(u).
\end{equation}
\begin{remark}
    We recall that the desingularized equation gives the flow in the correct direction for regions of the critical manifold where $h'(u)>0$, and one must reverse the flow in regions where $h'(u)<0$.
\end{remark}

We further notice that an equivalent expression for $\mathcal U_1$ is $\mathcal U_1=\left\{  u+C-Qh(u)=0\right\}$, provided we stay within the parameter's range leading to a unique equilibrium point. Next, let $u=\mathcal U_1$. The linearization of \eqref{eq:flujo3} at $\mathcal U_1$ is $1-Qh'(\mathcal U_1)$. That is, if $h'(\mathcal U_1)>\frac{1}{Q}$, then $\mathcal U_1$ is stable, and if $h'(\mathcal U_1)<\frac{1}{Q}$, then $\mathcal U_1$ is unstable. Due to its biological meaning, we are interested in the case where the equilibrium point $\mathcal U_1$ is stable for $h'(u)<0$ (i.e. on the attracting branch of the critical manifold). Given the desingularization process, we indeed verify that
\begin{enumerate}
    \item[i)] $g(u,h(u))<0$ when $u+C-Qh(u)<0$, that is, when $0<u<\mathcal{U}_1$,
    \item[ii)]  $g(u,h(u))>0$ when $u+C-Qh(u)>0$, that is, when $\mathcal{U}_1<u<1$,
    \item[iii)] $h'(u)>0$ when $0<\mathcal{U}_1<u_{p}$,
    \item[iv)] $h'(u)<0$ when $u_{p}<\mathcal{U}_1<1$.
\end{enumerate}
Therefore, we have the following cases
\begin{enumerate}
    \item If $h'(u)>0$
    \begin{equation*}
\frac{\dd u}{\dd \tau}\bigg|_{\mathcal M_0^1}
\left\{ \begin{array}{ll}  < 0, & \text{if} \;\; 0 < u < \mathcal{U}_1,\\ 
> 0, & \text{if} \;\;  \mathcal{U}_1<u<1. \end{array}\right.
\end{equation*}
\item If $h'(u)<0$
    \begin{equation*}
\frac{\dd u}{\dd \tau}\bigg|_{\mathcal M_0^1}
\left\{ \begin{array}{ll}  > 0, & \text{if} \;\; 0 < u < \mathcal{U}_1,\\ 
< 0, & \text{if} \;\;  \mathcal{U}_1<u<1. \end{array}\right.
\end{equation*}

\noindent These cases correspond to the sketches in Figure \ref{fig:slow1}.
\begin{figure}[htbp]
    \centering
    \includegraphics[width=0.3\linewidth]{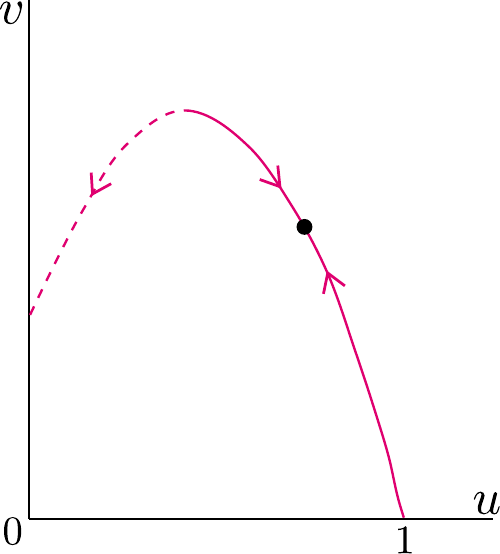}\hspace{2cm}
    \includegraphics[width=0.3\linewidth]{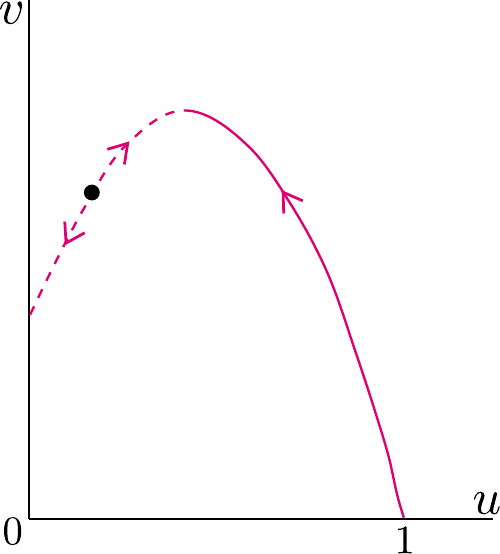}
    \caption{Slow flow when the equilibrium point $E_1$ is in either normally hyperbolic branch of $\mathcal M_0^1$.}
    \label{fig:slow1}
\end{figure}
%\hjk{Creo que los signos referentes a la ecuación anterior solo son correcta cuando tenemos un único punto de equilibrio. Creo que el flujo lento es como lo escribo en el pdf 'Note 17 Jul 2024' que se encuentra en la lista the archivos, en el menu de la izquierda. Tambi'n incluyo ahí una conjetura que quizás sea útil probar} \textcolor{green}{Ok. Ya quedó, solo para el caso en el que hay un único punto de equilibrio. Para el caso en el que hay más puntos de equilibrio tambien se puede demostrar tu conjetura ya que la dirección del flujo depende de los signos de $h'(u)$ y de $g(u,h(u))$, este último se hace cero en cada punto de equilibrio y por lo tanto cambia de signo}

\item When the equilibrium point $E_1$ coincides with the fold point $P$, that is $\mathcal{U}_{1}= u_p$, then $g(u,h(u))=0$ and $h'(u)=0$ for $Q=Q_{H}$. In this case, the right-hand side of \eqref{eq_flujo} is an indeterminate form. Thus, we apply L'Hospital's rule obtaining
\begin{equation}\label{eq:du}
\dfrac{\dd u}{\dd\tau}\bigg|_{u=u_p}=\lim_{u\rightarrow u_p}\frac{g(u,h(u))}{h'(u)}=\lim_{u\rightarrow u_p}\frac{g_{u}(u,h(u))}{h''(u)}=\frac{h(u)(u+A)}{h''(u)}<0,
\end{equation}
\end{enumerate}
which shows that, as sketched in Figure \ref{dinamicasf}, the flow in $\mathcal{M}_{0}^{1}$ is directed towards the fold point $P$ from the attracting side and away from it on the repelling side.

Now, we stress that Fenichel's theory \cite{Fenichel1979} implies that the manifold $\mathcal{M}_0^1$  is normally hyperbolic except at the points
$P$ and $T_{\mathcal C}$. Thus, any trajectory starting near the attracting (resp. repelling) submanifold $\mathcal{M}_{0}^{1,a}$ (resp. $\mathcal{M}_{0}^{1,r}$)  cannot cross the fold point $P$  (resp. $T_{\mathcal C}$). On the other hand, the manifolds 
 $\mathcal{M}_\varepsilon^{1,a}$ (resp. $\mathcal{M}_\varepsilon^{1,r}$) exist as a smooth perturbation of
 $\mathcal{M}_0^{1,a}$ (resp. $\mathcal{M}_0^{1,r}$). However, Fenichel's theory only guarantees the existence of such manifolds away from the fold point $P$ and $T_{\mathcal C}$. Our goal is to qualitatively describe the orbits of \eqref{eq_modelo1} in a small neighborhood of the singular points $P$ and $T_{\mathcal C}$ and to describe later on the global dynamics organized by such singularities. We are going to use a suitable, but singular, coordinate change known as 
 \emph{blow-up} \cite{kuehn2015multiple}, which will induce a new system with only (semi-)hyperbolic equilibrium points and, therefore, can be analyzed with standard tools from dynamical systems theory.
 
%

%\begin{figure}[h!]
%\centering
%\includegraphics[scale=0.5]{Figura44.pdf} 
%\caption{The dynamics of the slow subsystem \eqref{slowfast2} and of the fast  subsystem \eqref{slowfast3}.
%The graph shows the manifold $\mathcal{M}_{0}^{0}$ (red line) and the manifold $\mathcal{M}_{0}^{1}$ (blue line), corresponding 
%to the critical manifold (or slow manifold). 
%%${P}$ is the fold point and  $T_{\mathcal C}$ is the transcritical bifurcation point for the fast subsystem. 
%The horizontal flow lines correspond to the flow of the layer equation (rapid flow).}
%\label{dinamicasf}
%\end{figure}

We now turn our attention to the fold point $P=(u_p,v_p)$. Let us recall that a \emph{generic} fold point  is characterized by the following conditions 
$$
\frac{\partial f}{\partial u} (u_p,v_p)=0, \qquad \frac{\partial f}{\partial v} (u_p,v_p)\neq 0, \qquad \frac{\partial^2 f}{\partial u^2} (u_p,v_p)\neq 0, \quad\text{ and} \quad  g(u_p,v_p)\neq 0,$$
while a non-generic one further satisfies $g(u_p,v_p)=0$ \cite{Krupa2001ext}.
The local dynamics along ${\mathcal M}_0$ near the fold point $P$ can be distinguished between such two cases:
\begin{itemize}
    \item A generic fold point $P$ of the slow subsystem is singular and
solutions reach $P$ in finite forward or backward time. This case is known as jump point or regular contact point, which provides a condition for relaxation oscillations.
 \item A canard point is a fold point $P$ with an additional
degeneracy leading to a possibility of a canard solution.
 As a result, the system 
 has a solution passing through $P$ from the attracting to the repelling branch of the critical manifold $M_0$. 
\end{itemize}

The previous description is by now well-known, and we provide no further details but refer the reader to Figure \ref{fig:blowupfold} for a schematic of both cases, see also the textbooks \cite{kuehn2015multiple} and \cite{canardbook}.
\begin{figure}[htbp]
    \centering
    \begin{tikzpicture}
        \node at (0,0){
    \includegraphics{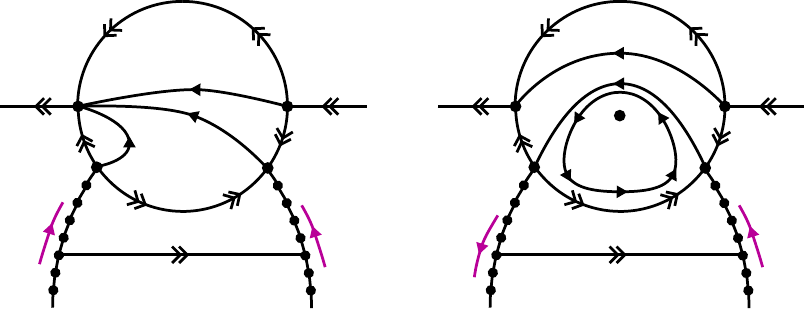}
        };
    \end{tikzpicture}
    \caption{Left: schematic of the blow-up of a generic fold point. Right: schematic of the blow-up of a canard point. For these pictures we have used the local setup of the model under study. Hence, the right branch of the critical manifold is attracting, while the left branch is repelling. In the left picture, the slow flow (magenta) is directed toward the fold point on both branches of the critical manifold. On the right picture, the slow flow ``passes through'' the fold point, compare with Figure \ref{dinamicasf}.}
    \label{fig:blowupfold}
\end{figure}

Usually, a canard point can alternatively be seen as a singular Hopf point. Whether the cycles that appear upon the bifurcation are stable or unstable depends on the criticality of the Hopf point. Until recently, determining the criticality requires putting the system into a normal form, which may sometimes be undesirable or unattainable. Recently, \cite{Maesschalck2021} proposes an intrinsic method to determine the criticality of a singular Hopf point; we shall apply such a method to our model in the following section and show that the canard point is indeed a (singular) supercritical Hopf point within a meaningful range of parameter values, which justifies the right-hand picture of Figure \ref{fig:blowupfold}. 

% %\textcolor{blue}{
% Actually, for the model under consideration \eqref{eq_modelo1}, Zhu and Liu \cite{Zhu2022} proved the occurrence of canard cycle and relaxation oscillations. Our goal in the following sections is to conduct research that complements what they have already done but in a new and meaningful way.
% Furthermore, we analyze the dynamics at the transcritical equilibrium point and its relationship with the canard cycles.}

% \hjk{Blend these two paragrpahs together}

We emphasize that Zhu and Liu \cite{Zhu2022} already proved the occurrence of canard cycles and relaxation oscillations in the generic scenario $C\neq-AMQ$. Our main objective in the rest of the document will be to show the existence of the singular cycles shown in Figure \ref{fig:singular-cycles} and to prove that such singular cycles perturb to nearby ones for $\ve>0$ sufficiently small.

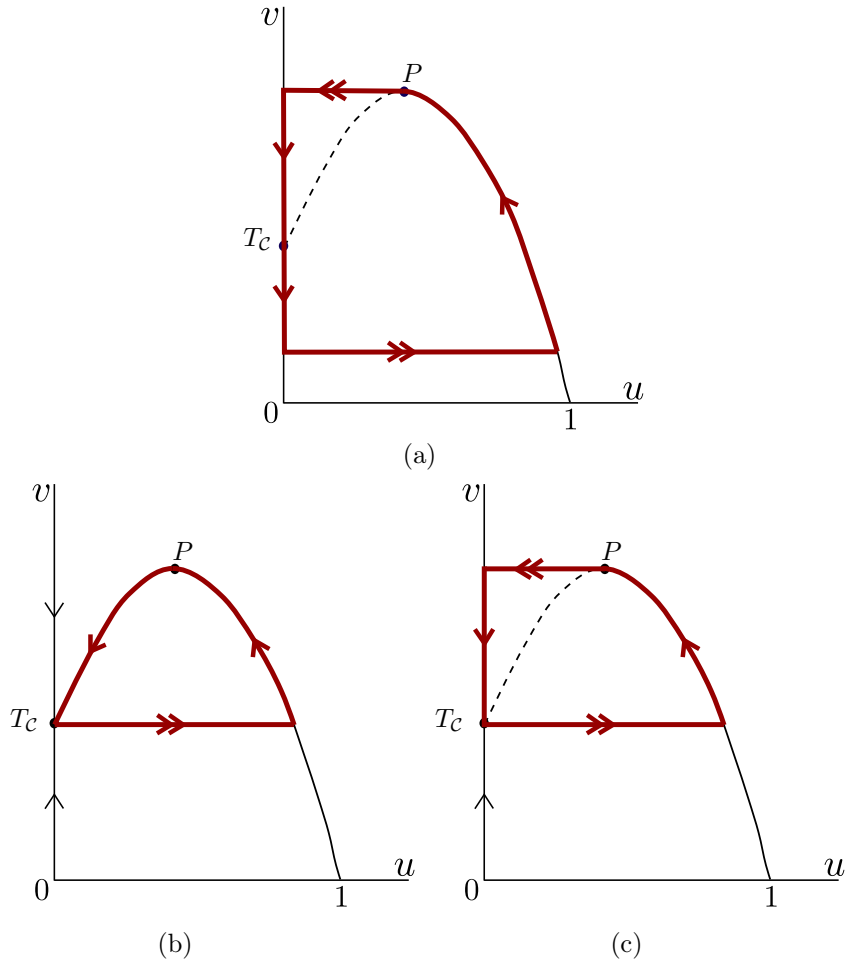
\begin{figure}[htbp]
    \centering
    \begin{tikzpicture}
        \node at (0,0){
        \includegraphics[scale=0.6]{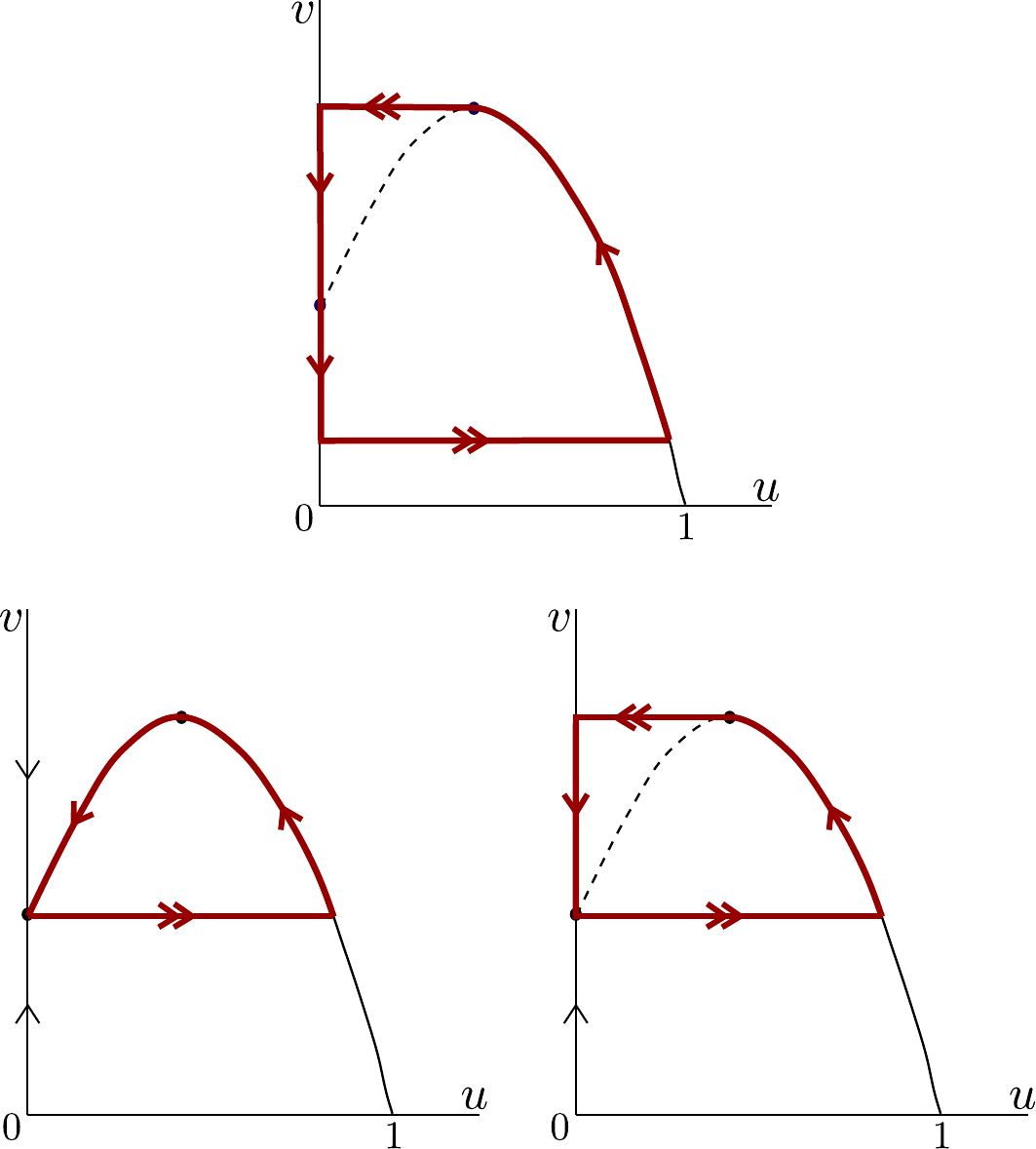}
        };
        \node at (-0.25,0){(a)};
        \node at (-3.5,-6.5){(b)};
        \node at ( 2.5,-6.5){(c)};
        \node at (-0.35,5.1){${P}$};
        \node at (-2.4,2.9){$T_{\mathcal C}$};
        \node at (-3.4,-1.25){${P}$};
        \node at (2.3,-1.25){${P}$}; 
        \node at (-5.5,-3.5){$T_{\mathcal C}$};
        \node at ( 0.15,-3.5){$T_{\mathcal C}$};
    \end{tikzpicture}
    \caption{Singular cycles that we focus on. The solid red cycle in (a) represents a candidate orbit for a regular relaxation oscillation, which appear when $C<-AMQ$. Such a cycle has already been studied in \cite{Zhu2022}, but here we include its analysis for completeness, see Section \ref{sec:relaxationoscillations}. The two on the bottom (b) and (c) are degenerate and, to the best of our knowledge, novel. At the singular level, they exist for $C=-AMQ$. As we show in Section \ref{sec_desing}, the one depicted in (b), called a transitory canard, is obtained by combining a maximal canard at $P$ and a (degenerate) saddle-like transition at $T_{\mathcal C}$. The one depicted in (c) combines a generic jump at $P$ and the same degenerate saddle-like transition at $T_{\mathcal C}$.}
    \label{fig:singular-cycles}
\end{figure}

\section{Intrinsic determination of the criticality of the singular Hopf bifurcation}\label{sec:intrinsic}

By using the inherent properties of the notion of a slow-fast Hopf point, De Maesschalck and co-authors \cite{Maesschalck2021}
prove that the criticality of a slow-fast Hopf can be determined without the necessity of having or changing to a normal form. The authors present a single intrinsic formula for an intrinsic determination that verifies if any singular contact point of Hopf type exists and its criticality.
Such an intrinsic formulation is beneficial because the critical curve does not require a parametrization. We use such a technique in the following to show that system \eqref{slowfast1} may undergo a supercritical Hopf bifurcation as the parameter $Q$ passes through $Q_H$.

Following \cite{Maesschalck2021}, in this section, we deal with a slow-fast Hopf point in \eqref{slowfast1} at $P=(u_p,v_p)$ for $Q=Q_H$. Through the analysis of this section, we show that within a meaningful range of parameters, such a Hopf point is supercritical. 
%The notation we use in the following is as close as possible to that in \cite{Maesschalck2021}.

Considering the system \eqref{slowfast1}, we define
\begin{equation*}
\begin{split}
{\mathcal F}(u,v)&= u(u+C)\big((u+A)(1-u)(u-M)-v \big),\\
{\mathcal Q}(u,v)&=v (u+A)(u-Qv+C).
\end{split}
\end{equation*}

The techniques of \cite{Maesschalck2021} apply to vector fields on $2$-manifolds, however, it suffices for us to take the standard Euclidean metric, and associated area form. 
%
% Let $g$ be a metric on the   Riemannian manifold ${\mathcal M}$, and let also be $\nabla$ the gradient corresponding to that metric,
% Let $\Omega$ be the area form associated to the Euclidean metric metric on ${\mathcal M}$.
Using the formulas and the same notation as in \cite{Maesschalck2021},  for  $Z=\dfrac{\partial}{\partial u}$,  we compute
\begin{multline*}
Z{\mathcal F}(u,v)=u(u+C)(A(1+M-2u)+(2-3u)u+M(2u-1))\\
+u((M-u)(u-1)(A+u)-v)+(u+C)((M-u)(u-1)(A+u)-v),
\end{multline*}
\begin{multline*}
ZZ{\mathcal F}(u,v)=2(-3Mu+6u^2+6Mu^2-10u^3+C(-M+3u+3Mu-6u^2)\\
+A(-M+C(1+m-3u)+3u+3Mu-6u^2)-v),
\end{multline*}
$$
ZZZ{\mathcal F}(u,v)=6(-M+C(1+M-4u)+A(1-C+M-4u)+4u+4Mu-10u^2).
$$

At this moment we point out that $P=(u_{p},v_{p})$ is a generic contact point, since one can check that, generically (for $Q\neq Q_H$)
$${\mathcal F}(P)=0, \quad Z{\mathcal F}(P)=0, \quad ZZ{\mathcal F}(P)\neq 0.$$

We now compute
$$\det({\mathcal Q},Z)=
\begin{vmatrix}
0 & 1\\
{\mathcal Q}(u,v) & 0
\end{vmatrix}
=-v(u+A)(u-Qv+C),$$
and
\begin{equation*}
\begin{split}
&\det(\nabla {\mathcal F},\nabla Z{\mathcal F})=
\begin{vmatrix}
\dfrac{\partial {\mathcal F}}{\partial u} & \dfrac{\partial Z{\mathcal F}}{\partial u} \vspace{0.3cm}\\
\dfrac{\partial {\mathcal F}}{\partial v} & \dfrac{\partial Z{\mathcal F}}{\partial v}
\end{vmatrix}\vspace{0.3cm} \\
& \hspace{2cm} =A(C^2M+2CMu +(2M+C(1-3C+M))u^2-8Cu^3-4u^4)\\ 
&  \hspace{2.6cm}  +C^2((3+3M-8u)u^2+v)+2u^2((2+2M-5u)u^2+v)\\ 
& \hspace{2.6cm}  +Cu(-u(M-8(1+M)u+19u^2)+2v).
\end{split}
\end{equation*}
These terms will be used shortly. It follows from \cite[Theorem 1]{Maesschalck2021} that
\begin{align*}
\mathcal{G}(u,v)&=\det({\mathcal Q},Z)\cdot \det(\nabla {\mathcal F},\nabla Z{\mathcal F})\\
&=-\Big((A+u)v(u-Qv+C)\big(A(C^2M+2CMu\\
&\qquad +(2M+C(1-3C+M))u^2-8Cu^3-4u^4)+C^2((3+3M-8u)u^2+v)\\
&\qquad +2u^2((2+2M-5u)u^2+v)+Cu(-u(M-8(1+M)u+19u^2)+2v)\big)\Big),
\end{align*}
and one obtains the value
\begin{align*}
{\mathcal A}(u,v)= & 
\frac{ZZZ\mathcal{F}(u,v)}
{\left[ ZZ{\mathcal F}(u,v)\right]^2}\\
= &\frac{3}{\Lambda} \Big[\big(-M+C(1+M-4u)+A(1-C+M-4u)+4u+4Mu-10u^2\Big],
\end{align*}
where 
\begin{multline*}
\Lambda =2(-3Mu+6u^2+6Mu^2-10u^3+C(-M+3u+3Mu-6u^2)\\ +A(-M+C(1+M-3u)+3u+3Mu-6u^2)-v)^2.
\end{multline*}
Now, we define   ${\mathcal V}=a_{1}\dfrac{\partial }{\partial u}+a_{2}\dfrac{\partial }{\partial v}$, where $a_{1}$ and  $a_{2}$  are arbitrary functions to be determined.
The next step is to see that the conditions for a slow-fast Hopf point are satisfied and also to use the intrinsically defined vector field ${\mathcal V}$ with which we will provide the formula for the criticality. 
 
 First, we should find $a_1$ and  $a_2$ as the solution to the system of equations
 ${\mathcal V}({\mathcal F})=0$ and ${\mathcal V}(Z{\mathcal F})=1$. In doing that we  obtain 
\begin{equation}
\label{eq63}
\mathcal{V}=\frac{-\frac{\partial {\mathcal F}}{\partial v}}{\det(\nabla {\mathcal F},\nabla Z{\mathcal F})}\frac{\partial}{\partial u}+\frac{\frac{\partial {\mathcal F}}{\partial u}}{\det(\nabla {\mathcal F},\nabla Z{\mathcal F})}\frac{\partial}{\partial v}.
\end{equation}
We remark that according to \cite{Maesschalck2021}, this solution is uniquely defined. Let us next compute
$${\mathcal V}{\mathcal F}(u,v)=\frac{-\frac{\partial {\mathcal F}}{\partial v}}{\det(\nabla {\mathcal F},\nabla Z{\mathcal F})}\frac{\partial }{\partial u}{\mathcal F}+\frac{\frac{\partial{\mathcal F}}{\partial u}}{\det(\nabla {\mathcal F},\nabla Z{\mathcal F})}\frac{\partial }{\partial v}{\mathcal F},$$
$$\mathcal{V}\mathcal{G}(u,v)=\frac{-\frac{\partial {\mathcal F}}{\partial v}}{\det(\nabla {\mathcal F},\nabla Z{\mathcal F})}\frac{\partial}{\partial u}\mathcal{G}+\frac{\frac{\partial {\mathcal F}}{\partial u}}{\det(\nabla {\mathcal  F},\nabla Z{\mathcal F})}\frac{\partial}{\partial v} \mathcal{G},$$
$$\mathcal{V}^{2}\mathcal{G}(u,v)=\frac{-\frac{\partial {\mathcal F}}{\partial v}}{\det(\nabla {\mathcal F},\nabla Z{\mathcal F})}\frac{\partial}{\partial u}\mathcal{V} \mathcal{G}+\frac{\frac{\partial {\mathcal F}}{\partial u}}{\det(\nabla {\mathcal F},\nabla Z{\mathcal F})}\frac{\partial}{\partial v}\mathcal{V} \mathcal{G}.$$

One can confirm that as required by \cite[Theorem 2]{Maesschalck2021}, we have
\begin{equation}
    \mathcal G|_{Q=Q_H}(P)=0, \; \mathcal V|_{Q=Q_H}(\mathcal G|_{Q=Q_H}(P))<0,\;  \frac{\partial}{\partial Q} \mathcal G|_{Q=Q_H}(P)\neq 0.
\end{equation}

Combining these equations, we find 
\begin{equation}\label{eq:sigma}
    \sigma=\frac{1}{2}\mathcal{V}^{2}\mathcal{G}(u_{p},v_{p})-\mathcal{V}\mathcal{G}(u_{p},v_{p}) \:\mathcal{A}(u_{p},v_{p}).
\end{equation}

For our particular model, determining analytically the sign of $\sigma$ is unfeasible due to its complexity. However, we present representative slices of the graph of $\sigma=\sigma(A,C,M)$ given by \eqref{eq:sigma} in Figure \ref{fig:sigma}, where it is evident that both regimes, super- and sub- critical exist. Nevertheless, the previous analysis proves the following proposition, which is specialized to the purposes of the paper.
\begin{prop}
There exists an open set of parameter values $(A,M,C)$ with $A\approx\frac{1}{2}$ and $M\approx-\frac{1}{10}$ for which \eqref{slowfast1} undergoes a supercritical Hopf bifurcation at $Q=Q_H$.
\end{prop}

\begin{remark}
    The numerical values $A=\frac{1}{2}$ and $M=-\frac{1}{10}$ are chosen simply for purposes of simulations. One can compare and complement this with Figure \ref{fig:bifurcation-full}.
\end{remark}

% Determining the precise value of $\sigma$ out of parameters is not feasible because the expressions are very large. To visualize and illustrate this criticality $\sigma$,
% we now perform a numerical example  using the following parametric set, for case $AM+\frac{C}{Q}=0$:
% $$ A=0.5, \quad C= 0.1052, \quad M=-0.1, \quad Q=2.104012,\quad \varepsilon = 0.05,$$
% where the system experiences a supercritical Hopf bifurcation at $u=0.53869979$, $v=0.30603456$, with $\sigma =-0.914906 < 0$. \hjk{will this part be improved?}
\begin{figure}[htbp]
    \centering
    \begin{tikzpicture}
        \node at (-4,0){
        \includegraphics{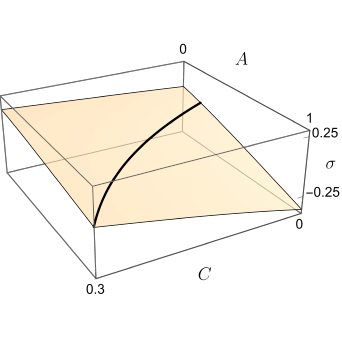}
        };
        \node at (4,0){
        \includegraphics{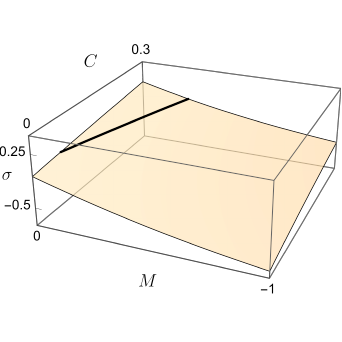}
        };

        \node at (-4,2.5){$M=-\frac{1}{10}$};
        \node at (4,2.5){$A=\frac{1}{2}$};
    \end{tikzpicture}
    
    \caption{We present slices of \eqref{eq:sigma}  for the indicated values of parameters. In this way, we verify the region of parameters for which the Hopf bifurcation is super- ($\sigma<0$) or sub- ($\sigma>0$) critical. The black curve indicates $\sigma=0$.}
    \label{fig:sigma}
\end{figure}

Now that we have determined that the point $P$ is a singular supercritical Hopf point, we proceed with the other singular point, which is the point $T_C$ under the condition that $C=-AMQ$.

\section{Desingularization of the point \texorpdfstring{$T_{\mathcal C}$}{TC}, in the degenerate case \texorpdfstring{$C=-AMQ$}{C=-AMQ}}\label{sec_desing}

In this section, we provide a local analysis of the system near the point $T_{\mathcal C}$ for the degenerate scenario $C=-AMQ$, recall Figure \ref{dinamicasf}. For convenience, we rewrite the system under consideration, namely
\begin{equation}\label{eq:sys_deg}
    X:\begin{cases}
        &\dfrac{\dd u}{\dd t}= u(u-AMQ)(AM - (u-M)(u-1)(A+u)-v),\vspace{0.1cm}\\
        &\dfrac{\dd v}{\dd t}=\ve(u+A)(v-AM)(u-Qv),\vspace{0.1cm}\\
        &\dfrac{\dd \ve}{\dd t}=0,
    \end{cases}
\end{equation}
where we have also shifted the origin to the point $T_{\mathcal C}$, recycled the coordinates, and added the trivial equation $\dfrac{\dd \ve}{\dd t}=0$. A simple computation confirms that the origin is nilpotent. Thus, we propose a blow-up of the form
\begin{equation}
    (u,v,\ve)=(r \bar u, r\bar v, r\bar\ve),
\end{equation}
where $\bar u^2+\bar v^2+\bar \ve^2=1$ and $r\in[0,r_0)$. Details of the blow-up transformation can be found in \cite{jardon2021} and references therein. As usual, we will study several local vector fields on different charts of the blow-up space. These charts and the corresponding local coordinates are introduced below.

For reasons that become clear in the blow-up analysis, \rev{see in particular  Propositions \ref{prop:K2}, \ref{prop:M2att}, \ref{prop:K3r0}, \ref{prop:K3e0}, \ref{prop:CM3} and Theorem \ref{teo:amq}, }we introduce the non-degeneracy conditions $A-M+AM\neq0$ and $A-M+AM\neq\frac{1}{Q}$.

\subsection{Entry chart}\label{sec:entry}
The local coordinates in this chart $K_1=\{\bar v=1\}$ are given by
\begin{equation}
    u=r_1u_1,\quad v=r_1,\quad \ve=r_1\ve_1,
\end{equation}
which leads to the local vector field
\begin{equation}\label{eq:K1e1}
    \begin{split}
        \frac{\dd r_1}{\dd t_1}&=r_1\ve_1(Q-u_1)(A^2M+\mathcal O(r_1)),\\
        \frac{\dd u_1}{\dd t_1}&=AMQ(1-(A+M-AM)u_1)u_1-A^2M(Q-u_1)u_1\ve_1+\mathcal O(r_1),\\
         \frac{\dd \ve_1}{\dd t_1}&=-\ve_1^2(Q-u_1)(A^2M+\mathcal O(r_1)),
    \end{split}
\end{equation}
where $t_1$ denotes the time in this chart (after blow-up and desingularization).

In this chart, we are first interested in the restriction of \eqref{eq:K1e1} to $r_1=0$, which reads as:
\begin{equation}\label{eq:K1e2}
    \begin{split}
          \frac{\dd u_1}{\dd t_1}&=AMQ(1-(A+M-AM)u_1)u_1-A^2M(Q-u_1)u_1\ve_1,\\
         \frac{\dd \ve_1}{\dd t_1}&=-A^2M\ve_1^2(Q-u_1).
    \end{split}
\end{equation}
\begin{prop} The origin is a semi-hyperbolic equilibrium point of \eqref{eq:K1e2} and has a unique $1$-dimensional, locally attracting, center manifold $\mathcal M_1$ associated to it. For $\ve_1\geq0$, $\mathcal M_1$ is unique, coincides with the $\ve_1$-axis, and the flow along it is directed away from the origin.
\end{prop}
\begin{remark}
    System \eqref{eq:K1e2} has another equilibrium point at $(u_1,\ve_1)=\left( \frac{1}{A-M+AM},0\right)$ but this is studied in the chart $K_3$.
\end{remark}
\begin{proof}
    The statements follow from standard center manifold arguments. We simply mention that the Jacobian at the origin is
    \begin{equation}
        J=\begin{bmatrix}
            AMQ & 0\\ 0 & 0
        \end{bmatrix},
    \end{equation}
    and recall that $M<0$.
\end{proof}

Next, the dynamics of \eqref{eq:K1e1} restricted to $\ve_1=0$ are given by:
\begin{equation}
    \begin{split}
         \frac{\dd r_1}{\dd t}&= 0,\\
         \frac{\dd u_1}{\dd t}&= -AMQ(-1+(A-M+AM)u_1)u_1+\mathcal \mathcal O(r_1u_1),
    \end{split}
\end{equation}
for which it is straightforward to see that the $r_1$-axis is a set of locally attracting equilibria. 
\begin{remark}
    Notice that in this chart $\ve=r_1\ve_1$. Therefore, rescaling time, it is possible to readily see from \eqref{eq:K1e1} that the slow flow along the $r_1$ axis is given by $r_1'=A^2MQ<0$. See Figure \ref{fig:K1}.
\end{remark}

Let the sections $\Delta_1^{\In}$ and $\Delta_1^{\out}$ (see Figure \ref{fig:K1}) be defined as follows:
\begin{equation}
    \begin{split}
        \Delta_1^{\In}&=\left\{ (r_1,u_1,\ve_1)\,:\,r_1=\delta_1,\,0\leq u_1<\tilde u_1, \, 0\leq \ve_1<\delta_2\right\}\\
        \Delta_1^{\out}&=\left\{ (r_1,u_1,\ve_1)\,:\ve_1=\delta_2,\,0\leq u_1<\tilde u_1, \, 0\leq r_1<\delta_1 \right\},
    \end{split}
\end{equation}
with $\delta_1,\,\delta_2$, and $\tilde u_1$ sufficiently small constants. Let $\Pi_1:\Delta_1^{\In}\to\Delta_1^{\out}$ be defined by the flow of \eqref{eq:K1e1}. From the analysis performed above, the next characterization of the map $\Pi_1$ follows:
\begin{prop}
    The map $\Pi_1$ is well-defined and is, moreover, a contraction. In particular, the image of $\Delta_1^{\In}$ under $\Pi_1$ is a wedge-like region contained in $\Delta_1^{\out}$.
\end{prop} 
The flow in this chart near the origin is sketched in Figure \ref{fig:K1}.
\begin{figure}[ht!]
    \centering
    \begin{tikzpicture}
    \node at (0,0){
        \includegraphics{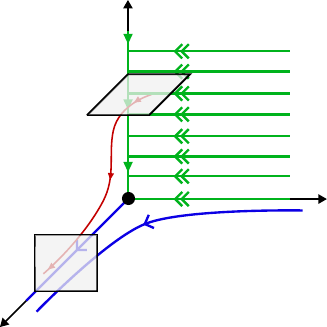}
        };
        \node at (-3,-3){$\ve_1$};
        \node at (3,-.65){$u_1$};
        \node at (-.5,3){$r_1$};
        \node[blue] at (-2.55,-2.1){\small$\mathcal M_1$};
        \node at (-1.5,1) {$\Delta_1^{\In}$};
        \node at (-2,-0.95) {$\Delta_1^{\out}$};
    \end{tikzpicture}
    \caption{Following the analysis presented in this section, we provide a sketch of the dynamics of \eqref{eq:K1e1} near the origin. The limit dynamics are presented in green and blue (for $\ve_1=0$ and $r_1=0$ respectively), while a sample orbit of \eqref{eq:K1e1} is shown in red.}
    \label{fig:K1}
\end{figure}
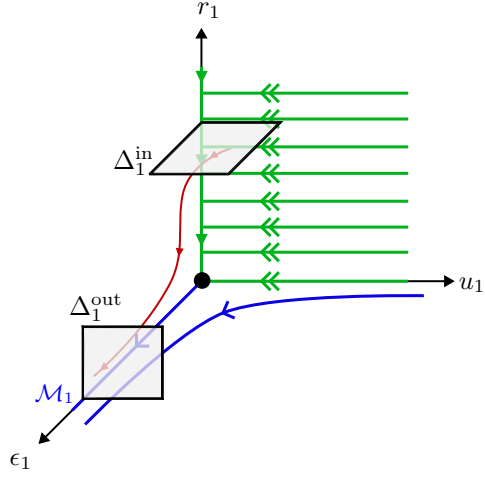

\subsection{Central chart}\label{sec:central}

The local coordinates in this chart $K_2=\{\bar\ve=1\}$ are given by
\begin{equation}
    u=r_2u_2,\quad v=r_2v_2,\quad \ve=r_2,
\end{equation}
which leads to the local vector field
\begin{equation}\label{eq:K2e1}
    \begin{split}
         \frac{\dd r_2}{\dd t_2}&=0\\
         \frac{\dd u_2}{\dd t_2}&=-u_2 (AMQ-r u_2) (A u_2 (M-r_2 u_2+1)+u_2 (r_2 u_2-1) (M-r_2 u_2)-v_2)\\
         \frac{\dd v_2}{\dd t_2}&=-((A+r_2 u_2) (u_2-Q v_2) (A M-r_2 v_2))\}.
    \end{split}
\end{equation}
Restricting to $\{r_2=0\}$ for \eqref{eq:K2e1}, we focus on
\begin{equation}\label{eq:K2e2}
    \begin{split}
         \frac{\dd u_2}{\dd t_2}&=-AMQ u_2((A-M+AM)u_2-v_2),\\
        \frac{\dd v_2}{\dd t_2}&=-A^2M(u_2-Qv_2),
    \end{split}
\end{equation}
where, since now $r_2$ is a regular parameter, the analysis that follows provides the qualitative dynamics also for $r_2$ sufficiently small.

\begin{prop}\label{prop:K2}
    The following statements hold for \eqref{eq:K2e2}:
    \begin{enumerate}
        \item The origin is the unique equilibrium point, and the $v_2$-axis is invariant under the flow. Moreover, the flow along the $v_2$-axis is directed towards the origin.
        \item The origin is semi-hyperbolic. For $u_2\geq0$, there exists a locally attracting $1$-dimensional center manifold $\mathcal M_2$ tangent, at the origin, to the center eigenspace $E^c=\textnormal{span}\left\{ \begin{bmatrix}
            Q\\1
        \end{bmatrix}\right\}$.
        \item For $u_2\geq0$: a) if $A-M+AM<\frac{1}{Q}$ then the origin is locally stable; b) if $A-M+AM>\frac{1}{Q}$ then the origin is unstable and the center manifold is unique.
        % \item Let $(u_2,v_2)\in\mathcal M_2$. If $A-M+AM<\frac{1}{Q}$ then $(u_2,v_2)\to(0,-\infty)$ as $t\to-\infty$. If  $A-M+AM>\frac{1}{Q}$, then $(u_2,v_2)\to(u_2^*,v_2^*)$, where $v_2\in(0,)$
    \end{enumerate}
\end{prop}
\begin{proof}
    The first two items are straightforward. For the third, let $\mathcal M_2=\{(u_2,v_2)\in\mathbb R^2 \,|\, v_2=h_2(u_2)\}$, with $h_2(0)=0$ and $\frac{dh}{\dd u_2}(0)=\frac{1}{Q}$. Standard center manifold computations lead to
    \begin{equation}
        h_2(u_2)=\frac{1}{Q}u_2+\frac{1-(A-M+AM)Q}{AQ^2}u_2^2+ \mathcal O(u_2^3),
    \end{equation}
    which implies that
    \begin{equation}
         \frac{\dd u_2}{\dd t_2}\biggm|_{\mathcal M_2}=-AMQ(A-M+AM-\tfrac{1}{Q})u_2^2+\mathcal O(u_2^3),
    \end{equation}
    leading to the statement (we recall that $M<0$). The reason why the center manifold is unique in the case $A-M+AM-\tfrac{1}{Q}>0$ is due to the saddle-like behavior near the origin.
\end{proof}

We are further interested in the case where the flow on the center manifold $\mathcal M_2$ goes away from the origin because this will help us prove the possibility of relaxation oscillations in the degenerate setting $C=-AMQ$. So, let us define the sections:
\begin{equation}
    \begin{split}
        \Delta_2^{\In}&=\left\{ (u_2,v_2)\,:\, 0\leq u_2<\tilde u_2,\, v_2=\delta_3\right\},\\
        \Delta_2^{\out}&=\left\{ (u_2,v_2)\,:\, u_2=\delta_4, |v_2-h_2(\delta_4)|<\delta_5\right\},
    \end{split}
\end{equation}
where $\tilde u_2,\delta_3,\delta_4,\delta_5$ are all positive small constants further satisfying $\tilde u_2<\delta_4$ and $h(\delta_4)+\delta_5<\delta_3$. Let $\Pi_2:\Delta_2^{\In}\to\Delta_2^{\out}$ be defined by the flow of \eqref{eq:K2e2}. The following proposition characterizes $\Pi_2$.
\begin{prop}\label{prop:M2att}
    Let $A-M+AM-\frac{1}{Q}>0$ and $\tilde u_2$ sufficiently small. The map $\Pi_2$ is well-defined. In particular, within a small neighborhood of $\mathcal M_2$, the flow of \eqref{eq:K2e2} contracts towards $\mathcal{M}_2$.
\end{prop}
\begin{proof}
    The statements follow from proposition \ref{prop:K2} and standard center manifold arguments. In particular, let $e=v_2-h$. It follows that:
    \begin{equation}
        \begin{split}
            \frac{\dd e}{\dd t_2} &= \frac{\dd v_2}{\dd t_2}-\frac{\partial h_2}{\partial u_2}\frac{\dd u_2}{\dd t_2}\\
            &=-A^2M(u_2-Q(e+h_2))-\frac{\partial h_2}{\partial u_2}\frac{\dd u_2}{\dd t_2}\\
            %&=A^2QMe-A^2Mu_2+A^2MQh_2-\frac{\partial h_2}{\partial u_2}\frac{\dd u_2}{\dd t_2}\\
             &=A^2QMe+A^2M(Qh_2-u_2)-\frac{\partial h_2}{\partial u_2}\frac{\dd u_2}{\dd t_2}\\
             &=A^2QMe,
        \end{split}
    \end{equation}
    where the last equality follows from the invariance of $\mathcal M_2=\left\{ v_2=h_2(u_2)\right\}$. Recall that $M<0$.
\end{proof}
We notice that Proposition \ref{prop:M2att} shows that, under the setting of the proposition, the center manifold $\mathcal M_2$ is locally attracting everywhere and not only close to the origin. 

The flow of \eqref{eq:K2e2} is sketched in Figure \ref{fig:K2}. We recall that in this chart, $r_2=\ve$ is a regular perturbation parameter.
\begin{figure}[htbp]
    \centering
    \begin{tikzpicture}
    \begin{scope}
    \clip(-5,-3.5) rectangle (5.5,3.5);
        \node at (0,0){
        \includegraphics{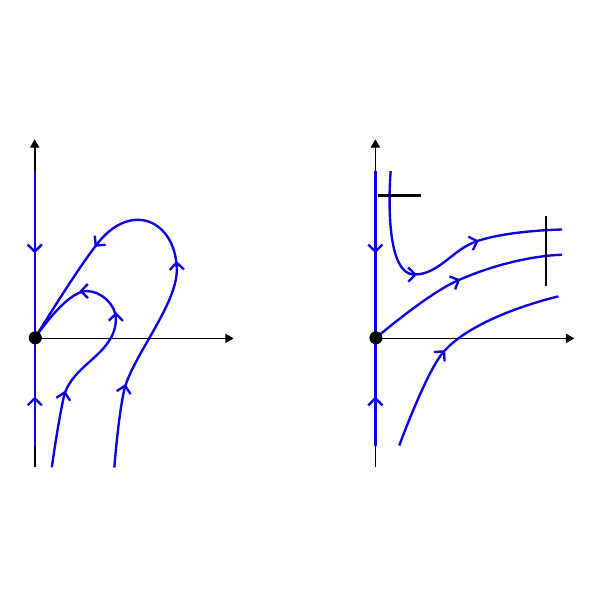}
        };
        \node at (-4.5,3){$v_2$};
        \node at (1.25,3){$v_2$};
        \node at (-0.9,-0.6){$u_2$};
        \node at (4.9,-0.6){$u_2$};
        \node at (2.5,1.75){$\Delta_2^{\In}$};
        \node at (4.5,1.5){$\Delta_2^{\out}$};
        \node[blue] at (3,0){$\mathcal M_2$};
    \end{scope}
        
    \end{tikzpicture}
    
    \caption{Phase-portrait of \eqref{eq:K2e2} for $A-M+AM-\frac{1}{Q}<0$ on the left and $A-M+AM-\frac{1}{Q}>0$ on the right. The phase portraits are drawn qualitatively by exploiting Proposition \ref{prop:K2}, the fact that $ \frac{\dd v_2}{\dd t}>0$ for all $v_2<0$ and the relative arrangement of the nullclines $v_2=(A-M+AM)u_2$, where we recall that $A-M+AM>0$, and $v_2=\frac{1}{Q}u_2$.}
    \label{fig:K2}
\end{figure}
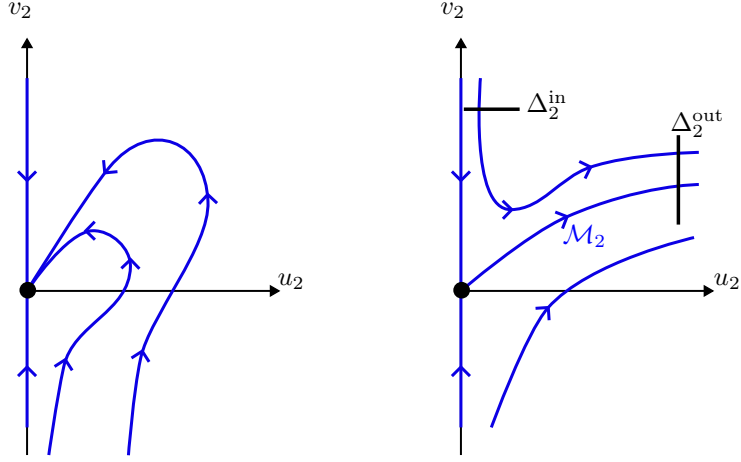

\subsection{Exit chart}\label{sec:exit}
The local coordinates in this chart $K_3=\{\bar u=1\}$ are given by
\begin{equation}
    u=r_3,\quad v=r_3v_3,\quad \ve=r_3\ve_3,
\end{equation}
which leads to the local vector field
\begin{equation}\label{eq:K3e1}
    \begin{split}
         %r_3'&=-r_3 (A M Q-r_3) (A (M-r_3+1)+(r_3-1) (M-r_3)-v_3) \\
         \frac{\dd r_3}{\dd t_3}&=-AMQr_3(A-M+AM-v_3)\\
         &\;+(A-M+AM-AMQ+A^2MQ-AM^2Q-v_3)r_3^2+\mathcal O(r_3^3),\\
         \frac{\dd v_3}{\dd t_3}&=AMQ(A-M+AM-v_3)v_3+A^2M(Qv_3-1)\ve_3+\mathcal O(r_3),\\
         \frac{\dd \ve_3}{\dd t_3}&=AMQ\ve_3(A-M+AM-v_3)+\mathcal O(\ve_3r_3),
 %v_3'&=A^2 M Q v_3 - A M^2 Q v_3 + A^2 M^2 Q v_3 + r_3^3 v_3 - A M Q v_3^2 + \ve_3 (-A^2 M + A^2 M Q v_3) + r_3 (-A v_3 + M v_3 - A M v_3 + A M Q v_3 - A^2 M Q v_3 + A M^2 Q v_3 + v_3^2 + A \ve_3 (M - v_3) (-1 + Q v_3)) + r_3^2 (-v_3 + A v_3 - M v_3 - A M Q v_3 + \ve_3 (v_3 - Q v_3^2)) \\
%\ve_3'&= \ve_3 (A M Q-r_3) (A (M-r_3+1)+(r_3-1) (M-r_3)-v_3) \\
    \end{split}
\end{equation}
where $t_3$ denotes the rescaled time in this chart and $r_3\geq0$, $\ve_3\geq0$. First, we consider \eqref{eq:K3e1} restricted to $r_3=0$, that is:
\begin{equation}\label{eq:K3e2}
    \begin{split}
         \frac{\dd v_3}{\dd t_3}&= AMQ(A-M+AM-v_3)v_3+A^2M(Qv_3-1)\ve_3,\\
        \frac{\dd \ve_3}{\dd t_3}&=AMQ\ve_3(A-M+AM-v_3).
    \end{split}
\end{equation}
\begin{prop}\label{prop:K3r0}
    The following statements hold for \eqref{eq:K3e2} (recall that $A-M+AM>0$):
    \begin{enumerate}
        \item There are two equilibrium points:
        \begin{equation}
            o_3=(0,0)\qquad \textnormal{and}\qquad p_3=(A-M+AM,0).
        \end{equation}
        \item The equilibrium point $o_3$ is a hyperbolic sink.
        
        %If $A-M+AM>0$ then $o_3$ is a hyperbolic sink, and if $A-M+AM<0$ then $o_3$ is a hyperbolic source. \hjk{The algebraic multiplicity of the eigenvalue is $2$ and the geometric multiplicity is $1$. I am not sure if this is going to be relevant for the analysis, but I leave this note here so that I don't forget}
        \item The equilibrium point $p_3$ is semi-hyperbolic and has a $1$-dimensional locally repelling center manifold $\mathcal N_3$ associated to it. If in addition $A-M+AM<\frac{1}{Q}$ then the flow along $\mathcal N_3$ is directed towards $p_3$ and $\mathcal N_3$ is unique; on the other hand, if $A-M+AM>\frac{1}{Q}$ then the flow along $\mathcal N_3$ is directed away from $p_3$.
        % \begin{enumerate}
        %     \item If $A-M+AM>0$ then $\mathcal M_3$ is locally repelling. If in addition $A-M+AM<\frac{1}{Q}$ then the flow along $\mathcal M_3$ is directed towards $p_3$ and $\mathcal M_3$ is unique; on the other hand, if $A-M+AM>\frac{1}{Q}$ then the flow along $\mathcal M_3$ is directed away from $p_3$.
        %     \item  If $A-M+AM<0$ then $\mathcal M_3$ is locally attracting. Moreover, the flow along $\mathcal M_3$ is directed away from $p_3$ and $\mathcal M_3$ is unique.
        % \end{enumerate}
    \end{enumerate}
\end{prop}
\begin{proof}
    The first two statements are straightforward. For the third, standard center manifold computations show that $ \mathcal N_3=\{(v_3,\ve_3)\in\mathbb R^2\,|\,v_3=g_3(\ve_3)\}$ is given by
    \begin{equation}
       g_3(\ve_3)=A-M+AM+\frac{A(A-M+AM-\frac{1}{Q})}{(A-M+AM)Q}\ve_3+\frac{A^2(A-M+AM-\frac{1}{Q})}{(A-M+AM)Q^2}\ve_3^2+\mathcal O(\ve_3^3),
    \end{equation}
    which implies that 
    \begin{equation}
        \frac{\dd \ve_3}{\dd t_3}\biggm|_{\mathcal N_3}=-\frac{A^2M(A-M+AM-\frac{1}{Q})}{(A-M+AM)}\ve_3^2+\mathcal O(\ve_3^3),
    \end{equation}
    leading to the statement (we recall that $M<0$). The uniqueness of the center manifold follows from the saddle-like behavior in the corresponding case.
\end{proof}

\begin{remark}
    The center manifold $\mathcal M_2$ in chart $K_2$ is contained in the forward invariant set $\Gamma_2=\left\{ (u_2,v_2)\in\mathbb R^2\,|\, u_2>0, \, 0<v_2<\frac{1}{Q}u_2 \right\}$. In the coordinates of chart $K_3$, $\Gamma_2$ corresponds to $\Gamma_3=\left\{ (\ve_3,v_3)\in\mathbb R^2\,|\,  0<v_3<\frac{1}{Q} \right\}$. It then follows from Proposition \ref{prop:K3r0} that for $A-M+AM>\frac{1}{Q}$, the unique center manifold $\mathcal M_2$ converges to the origin of chart $K_3$. In this setting, let us denote by $\mathcal M_3$ the parametrization of $\mathcal M_2$ in the coordinates of $K_3$ and assume (at least locally) $\mathcal M_3=\left\{ v_3=h_3(\ve_3) \right\}$.
\end{remark}

On the other hand, system \eqref{eq:K3e1} restricted to $\ve_3=0$ is given by:
\begin{equation}\label{eq:K3e4}
    \begin{split}
        \frac{\dd r_3}{\dd t_3}&= -(AMQ-r_3)(A(1+M-r_3)+(M-r_3)(r_3-1)-v_3)r_3,\\
       \frac{\dd v_3}{\dd t_3}&= (AMQ-r_3)(A(1+M-r_3)+(M-r_3)(r_3-1)-v_3)v_3,
    \end{split}
\end{equation}
where we recall that $r_3\geq0$.

% If, for a moment, we think of $\ve_3$ as a parameter, we notice that \eqref{eq:K3e4} can be regarded as a slow-fast system in non-standard form that can be written as:
% \begin{equation}
%     z'=N(z)f(z),\quad z=(r_3,v_3),
% \end{equation}
% where
% \begin{equation}
%     \begin{split}
%         N(z)&=\begin{bmatrix}
%             -r_3\\v_3
%         \end{bmatrix},\\
%         f(z)&=(AMQ-r_3)(A(1+M-r_3)+(M-r_3)(r_3-1)-v_3).
%     \end{split}
% \end{equation}
\begin{prop}\label{prop:K3e0}
    The following statements hold for \eqref{eq:K3e4} (recall that $A-M+AM>0$):
    \begin{enumerate}
        \item The origin $o_3=(0,0)$ is a hyperbolic saddle. In particular, the $r_3$-direction is repelling while the $v_3$-direction is attracting. 
        \item There is a curve of equilibria
        \begin{equation}
            \ell_3=\left\{ (r_3,v_3)\in\mathbb R^2\,|\, v_3=A-M+AM+(1-A+M)r_3-r_3^2\right\}.
        \end{equation}
        For $r_3$ sufficiently small, the curve $\ell_3$ is repelling.
    \end{enumerate}
\end{prop}

\begin{proof}
    All statements are straightforward. We indicate that the nontrivial eigenvalue along $\ell_3$ is $\lambda=-AMQ(A-M+AM)+\mathcal O(r_3)$ and recall that $M<0$.
\end{proof}

\begin{remark}
    It follows by substituting $(u,v)=(r_3,r_3v_3)$ into the expression of the critical manifold of \eqref{eq:sys_deg}, that the curve $\ell_3$ is, in fact, the critical manifold $\mathcal M_0^1$ written in the coordinates of chart $K_3$.
\end{remark}

\begin{prop}\label{prop:CM3}
    Let $A-M+AM-\frac{1}{Q}>0$. There is a (\rev{not necessarily unique}) $2$-dimensional center manifold $\mathcal P_3$ at the point $p_3$. The flow restricted to it is of saddle type, with the direction $r_3$ stable and the direction $\ve_3$ unstable, as sketched in Figure \ref{fig:K3}.
\end{prop}
\begin{proof}
    The result follows from standard center manifold computations, so we simply report the important steps. Assuming that the center manifold can be given as the graph of $v_3=H_3(r_3,\ve_3)$, we find \begin{equation}
        H_3=A-M+AM+\frac{-1+(A-M+AM)Q}{(A-M+AM)Q}A\ve_3+(1-A+M)r+\mathcal{O}(r_3^2,\ve_3^2,r_3\ve_3).
    \end{equation}
    The reduced dynamics on the center manifold are therefore given by
    \begin{equation}
        \begin{split}
            \frac{\dd r_3}{\dd t_3}&=\frac{r_3(AMQ-r_3)(A\ve_3(-1+(A-M+AM)Q)+(A-M+AM)Qr^2)}{(A-M+AM)Q},\\
            \frac{\dd \ve_3}{\dd t_3}&=-\frac{\ve_3(AMQ-r_3)(A\ve_3(-1+(A-M+AM)Q)+(A-M+AM)Qr^2)}{(A-M+AM)Q}.
        \end{split}
    \end{equation}
    The condition $A-M+AM-\frac{1}{Q}>0$, together with $M<0$ and $A-M+AM>0$, imply that the common term $\frac{A\ve_3(-1+(A-M+AM)Q)+(A-M+AM)Qr^2}{(A-M+AM)Q}$ is positive. \rev{In addition, recalling Proposition \ref{prop:K3r0}, we know that $p_3$ is a non-hyperbolic source, suggesting the non-uniqueness of the center manifold. }The statement therefore follows.
\end{proof}

We now describe the two most important transition maps. Let us define the sections:
\begin{equation}
    \begin{split}
        \Delta_3^{\In}&=\left\{ (r_3,v_3,\ve_3)\,:\, \ve_3=\delta_6,\, |v_3-h_3(\delta_6)|<\delta_7, \, 0\leq r_3<\tilde r_3 \right\},\\
        \Delta_c^{\In}&=\left\{ (r_3,v_3,\ve_3)\,:\, r_3=\rho_3,\, |v_3-\ell_3|<e^{-c/\ve}, \,0\leq\ve_3<\delta_{8}\right\}\\
        \Delta_3^{\out}&=\left\{ (r_3,v_3,\ve_3)\,:\, r_3=\delta_9, \, |v_3|<\delta_{10},\,0\leq\ve_3<\tilde \ve_3 \right\},
    \end{split}
\end{equation}
where all the introduced constants are sufficiently small and positive, further satisfying $\tilde\ve_3<\delta_6$, $\tilde r_3<\delta_9$, $\delta_8<\delta_6$. Let the maps $\Pi_3:\Delta_3^{\In}\to\Delta_3^{\out}$ and $\Pi_3^c:\Delta_c^{\In}\to\Delta_3^{\out}$ be defined by the flow of \eqref{eq:K3e1}. The analysis presented above suffices to show the following:
\begin{prop}\label{prop:K3maps}$ $
\begin{enumerate}
\item The map $\Pi_3$ is well defined. In particular, the image of $\Delta_3^{\In}$ under the map $\Pi_3$ is a wedge-like region contained in $\Delta_3^{\out}$ and $\Pi_3$ is an exponential contraction towards the $r_3$-axis.
\item The map $\Pi_3^c$ is well defined. The image of the exponentially thin strip $\Delta_c^{\In}$ is an exponentially thin strip contained in $\Delta_3^{\out}$.
\end{enumerate}  
\end{prop}

The dynamics in this chart, following Propositions \ref{prop:K3r0}, \ref{prop:K3e0}, \ref{prop:CM3} and \ref{prop:K3maps}, are summarized in Figure \ref{fig:K3}.
\begin{figure}[htbp]
    \centering
    \begin{tikzpicture}
        \node at (0,0){
        \includegraphics{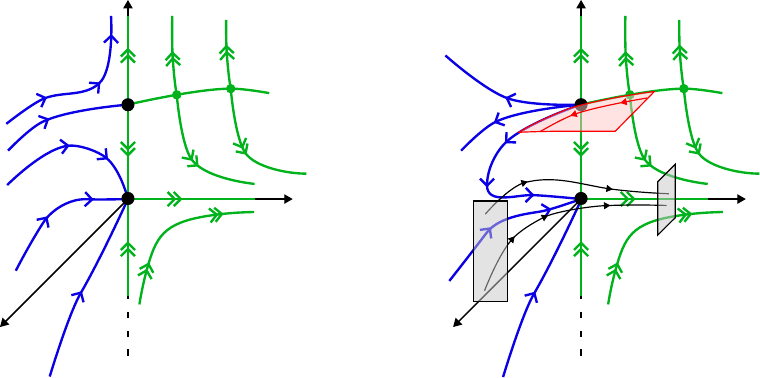}
        };
        \node at (-6.5,-2.5){$\ve_3$};
        \node at ( 1.25,-2.5){$\ve_3$};

        \node at ( 3.4,3.35){$v_3$};
        \node at (-4.3,3.35){$v_3$};

        \node at ( 6.45,-.2){$r_3$};
        \node at (-1.2,-.2){$r_3$};

        \node[green!75!black] at (-1.6,1.55){$\ell_3$};
        \node[green!75!black] at (6,1.55){$\ell_3$};

        \node[blue] at (0.75,-1.5) {$\mathcal M_3$};
        \node at (2,-2.25){$\Delta_3^{\In}$};
        \node at (5,-0.95){$\Delta_3^{\out}$};
    \end{tikzpicture}
    \caption{Dynamics of \eqref{eq:K1e1} in the chart $K_3=\{\bar u=1\}$ for $A-M+AM-\frac{1}{Q}<0$ on the left and $A-M+AM-\frac{1}{Q}>0$ on the right. The flow in blue occurs in the $(\ve_3,v_3)$-plane, while the green one in the $(r_3,v_3)$-plane. The black orbits are a sketch of a sample of orbits flowing from $\Delta_3^{\In}$ and arriving at $\Delta_3^{\out}$. The red surface is the center manifold $\mathcal P_3$. Notice that since $\ell_3$ is simply the critical manifold $\mathcal M_0^1$ written in the local coordinates of this chart, the orbits on $\mathcal P_3$ can be identified with the slow manifolds that perturb from $\mathcal M_0^1$. Although we do not show $\Delta_c^{\In}$, as it is an exponentially thin strip, it is evident from the analysis presented that some of the orbits near $\mathcal P_3$ arrive at $\Delta_3^{\out}$.}
    \label{fig:K3}
\end{figure}
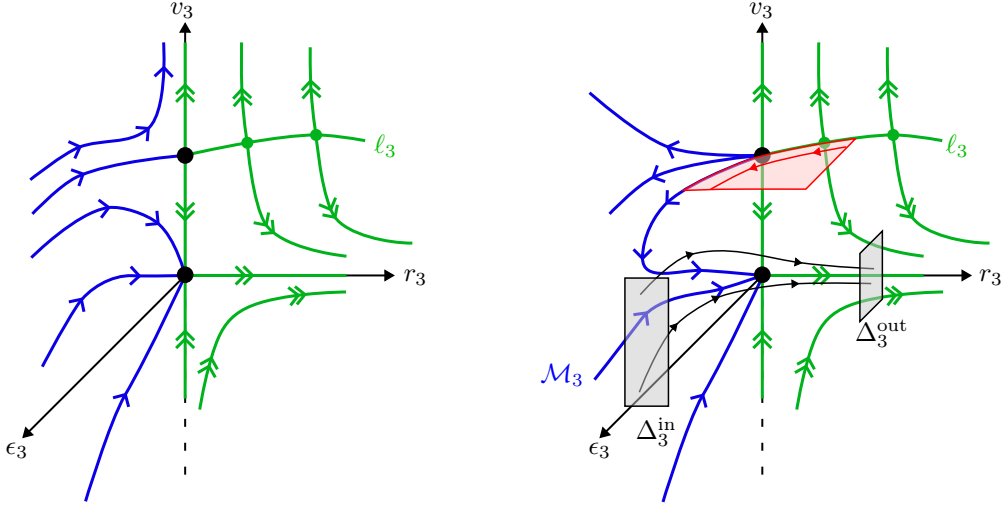

For completeness, we briefly present the dynamics on the bottom chart in the following section.

\subsection{Bottom chart}\label{sec:bottom}
The local coordinates in this chart $K_4=\{\bar v=-1\}$ are given by
\begin{equation}
    u=r_4u_4,\quad v=-r_4,\quad \ve=r_4\ve_4,
\end{equation}
which leads to the local vector field
\begin{equation}\label{eq:K4e1}
    \begin{split}
       \frac{\dd r_4}{\dd t}&=r_4\ve_4(Q+u_1)(A^2M+\mathcal O(r_4)),\\
       \frac{\dd u_4}{\dd t}&=-AMQ(1+(A+M-AM)u_4)u_4-A^2M(Q+u_4)u_4\ve_4+\mathcal O(r_4),\\
        \frac{\dd \ve_4}{\dd t}&=-\ve_4^2(Q+u_4)(A^2M+\mathcal O(r_4)).
    \end{split}
\end{equation}
The dynamics are similar to that on the entry chart, just with the $\ve$-direction reversed, so we simply show in Figure \ref{fig:K4} the corresponding flow and omit the details.
\begin{figure}
   \centering
    \begin{tikzpicture}
    \node at (0,0){
        \includegraphics{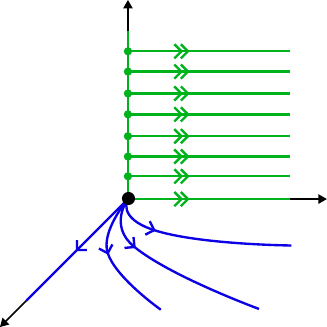}
        };
        \node at (-3,-3){$\ve_1$};
        \node at (3,-.65){$u_4$};
        \node at (-.5,3){$r_4$};
        %\node at (-1.5,-1){$\mathcal M_1$};
    \end{tikzpicture}
    \caption{Dynamics of \eqref{eq:K4e1} near the origin. In contrast to the flow of \eqref{eq:K1e1}, the center manifold in the plane $(v_4,\ve_4)$ is not unique.}
    \label{fig:K4}
\end{figure}
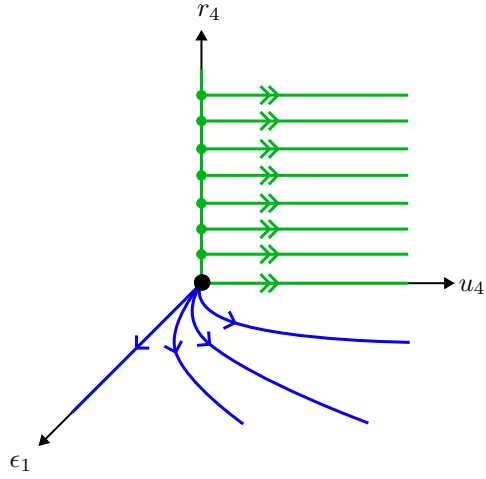

\subsection{Blown-up dynamics at \texorpdfstring{$T_{\mathcal C}$}{TC}, relaxation oscillations, and transitory canard}
With the analysis performed in sections \ref{sec:entry}-\ref{sec:bottom} we deduce that the blown-up dynamics at $T_{\mathcal C}$ are given by two non-equivalent flows as shown in Figure \ref{fig:blowupTC}.
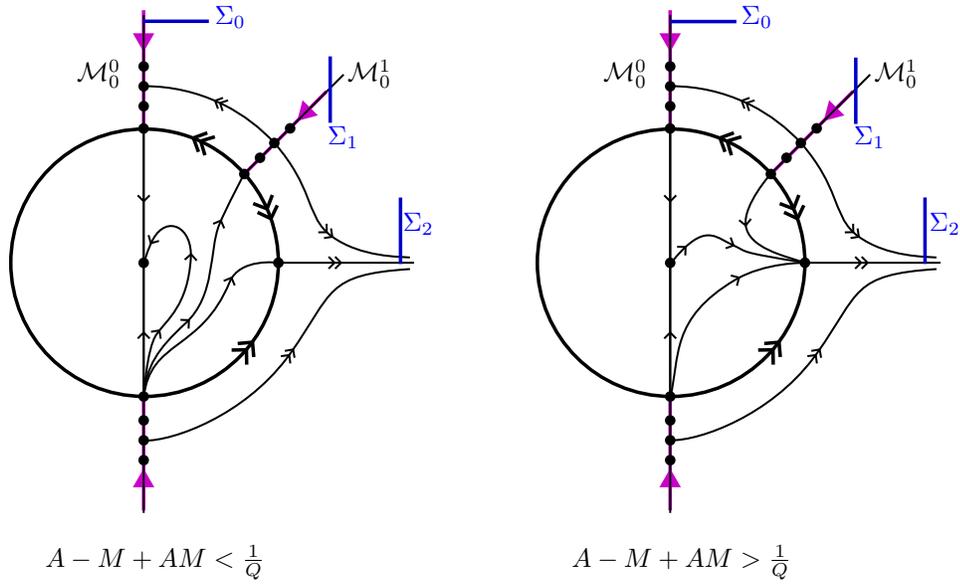
\begin{figure}[htbp]
    \centering
    \begin{tikzpicture}
        \node at (0,0){
        \includegraphics{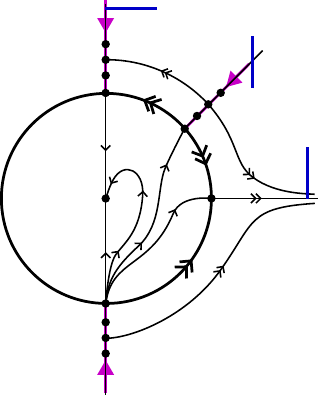}
        };
        \node at (-1.5,2.5){$\mathcal{M}_0^0$};
        \node at (2.1,2.5){$\mathcal{M}_0^1$};
        \node[blue] at (0.25,3.25){$\Sigma_0$};
        \node[blue] at (1.75,1.65){$\Sigma_1$};
        \node[blue] at (2.75,.5){$\Sigma_2$};
        \node at (-0.75,-4){$A-M+AM<\frac{1}{Q}$};
        \node at (7,0){
        \includegraphics{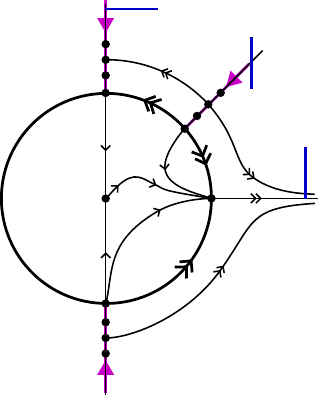}
        };
        \node at (5.5,2.5){$\mathcal{M}_0^0$};
        \node at (9.1,2.5){$\mathcal{M}_0^1$};
        \node[blue] at (7.25,3.25){$\Sigma_0$};
        \node[blue] at (8.75,1.65){$\Sigma_1$};
        \node[blue] at (9.75,.5){$\Sigma_2$};
        \node at (6.25,-4){$A-M+AM>\frac{1}{Q}$};
    \end{tikzpicture}
    \caption{The two non-equivalent blown-up dynamics of the point $T_{\mathcal C}$ for $A-M+AM<\frac{1}{Q}$ on the left and $A-M+AM>\frac{1}{Q}$ on the right. The dotted curves depict the critical manifold. The corresponding reduced flows are indicated by the magenta arrows. Compare with Figure \ref{dinamicasf} near the point $T_\mathcal{C}$.}
    \label{fig:blowupTC}
\end{figure}

In particular, the proof of the following proposition follows from the analysis carried out previously.

\begin{prop}\label{prop:singtr}
    Consider system \eqref{eq:sys_deg} and let $\Sigma_0$, $\Sigma_1$, and $\Sigma_2$ be sections defined as follows (refer also to Figure \ref{fig:blowupTC}):
    \begin{equation}
        \begin{split}
            \Sigma_0 &= \left\{ (u,v,\ve)\in\mathbb R^3\,:\, 0\leq u<\tilde u,\, v=\tilde v,\,\ve\ll1 \right\} \\
            \Sigma_1 &= \left\{ (u,v,\ve)\in\mathbb R^3\,:\, u=u^*,\,h(u)-\delta\leq v< h(u)+\delta,\,\ve\ll1 \right\}\\
            \Sigma_2 &= \left\{ (u,v,\ve)\in\mathbb R^3\,:\, u=u^\star,\,0\leq v< h(u),\,\ve\ll1 \right\}, 
        \end{split}
    \end{equation}
    where $\tilde u<u^*<u^\star$ are sufficiently small positive constants. 
    Then, for $A-M+AM>\frac{1}{Q}$ and $\ve>0$ sufficiently small:
    \begin{enumerate}
        \item The map $\Pi_{0\to2}:\Sigma_0\mapsto\Sigma_2$ induced by the flow of \eqref{eq:sys_deg} is well defined and is a contraction.
        \item For $\delta>0$ exponentially small and $Q=Q_m\sim Q_H$ as $\ve\to0$, where $Q_m$ is the value of the parameter $Q$ for which there is a maximal canard at $P$,  the map $\Pi_{1\to2}:\Sigma_1\mapsto\Sigma_2$ induced by the flow of \eqref{eq:sys_deg} is well defined. 
    \end{enumerate}
\end{prop}
\begin{proof}
    At the singular level, i.e. for $\ve=0$, the map $\Pi_{0\to2}$ is described by the concatenation of the corresponding orbits for $\bar r=0$ in the blow-up analysis performed in Sections \ref{sec:entry}-\ref{sec:bottom}, refer in particular to Figures \ref{fig:K1}, \ref{fig:K2} and \ref{fig:K3} and the analysis leading to them, together with a blow-down. For the map $\Pi_{1\to2}$ one just needs to be careful to have the appropriate parameters such that the slow flow along $\mathcal M_{0}^{1,r}$ is directed toward $T_\mathcal{C}$ as in Figure \ref{fig:singular-cycles} and that the parameter $Q$ is chosen so that there is a maximal canard at the fold point $P$. This is achieved for parameters exponentially close to $Q=Q_H$. Since we have already shown in Section \ref{sec:intrinsic} that the singular Hopf bifurcation at $P$ for $Q=Q_H$ may be supercritical, the appropriate choice of parameters guarantees that such a maximal canard is stable for $\ve>0$ sufficiently small. 
\end{proof}

\begin{remark}
    For $A-M+AM<\frac{1}{Q}$, one can show that the trajectories that cross $\Sigma_0$ transversely converge to the origin. Since this does not lead to oscillatory behavior, we do not detail this case anymore. However, it is interesting to observe that, in this case, the degenerate point $T_\mathcal{C}$ is a saddle but eventually attracts all orbits.
\end{remark}

Next, we discuss oscillations that are organized by the two singularities, the fold $P$ and the degenerate saddle $T_\mathcal{C}$, see Figure \ref{fig:DegenerateCycles},  which is now presented in the original coordinates $(u,v)$ of system \eqref{eq_modelo1} where $S=\ve$.

\begin{figure}[htbp!]
    \centering
    \includegraphics[scale=0.5]{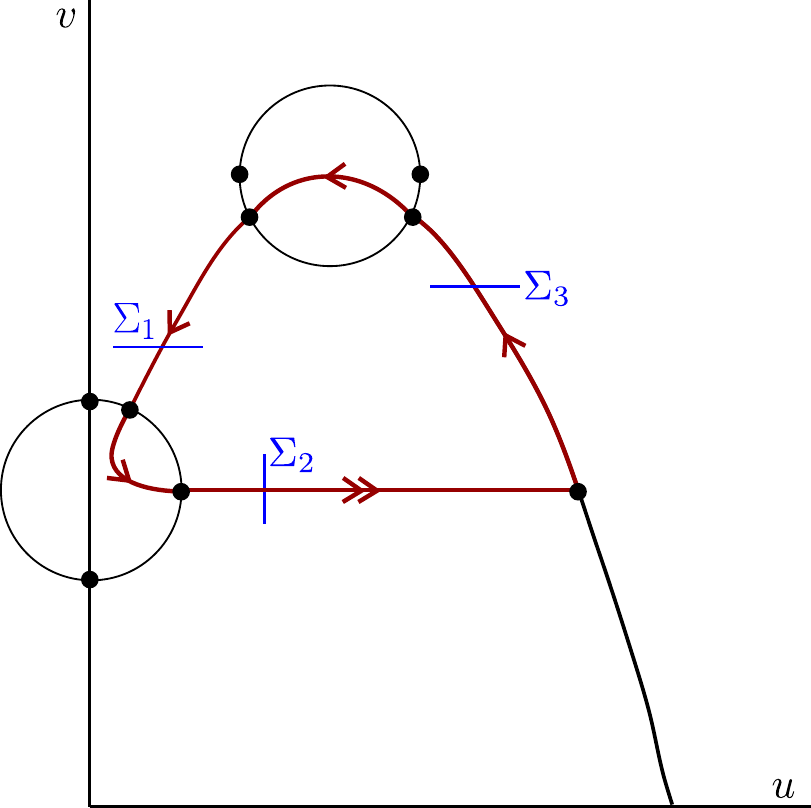}\hfill
    \includegraphics[scale=0.5]{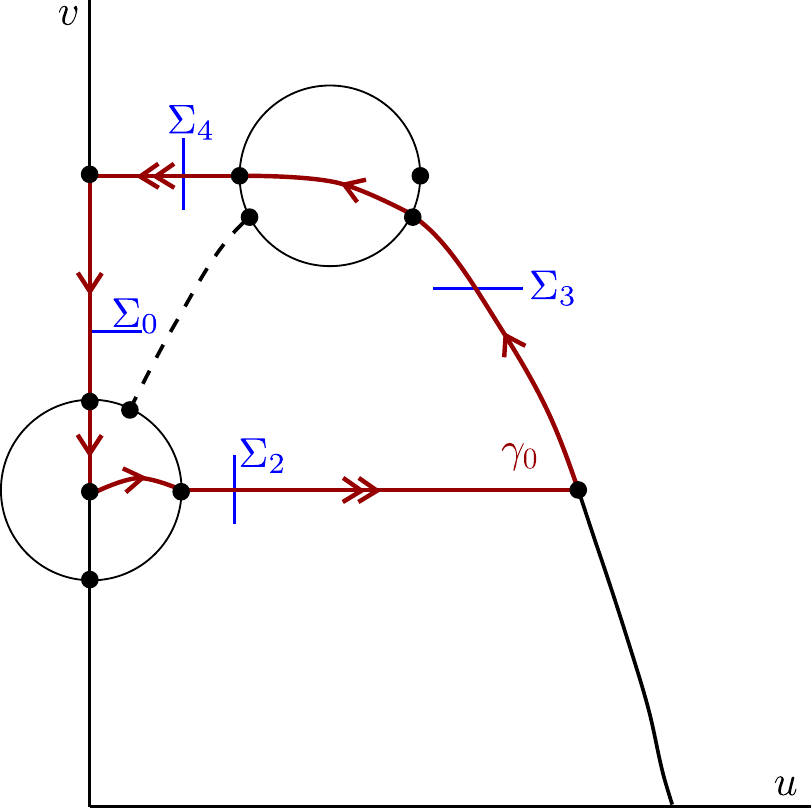}
    \caption{On the left we show we show a transitory canard passing through the fold point $P$ and the singular point $T_\mathcal{C}$. On the right, a singular relaxation oscillations passing through the same singularities. These cycles are obtained for different parameter regimes (recall that in this section $C=-AMQ$) provided that $A-M+AM>\frac{1}{Q}$. The relaxation oscillation is obtained when the fold point $P$ is a generic jump point. The transitory canard is obtained when the fold point undergoes a singular Hopf bifurcation and the parameter $Q\approx Q_H$ is chosen so that there is a maximal canard at $P$.}
    \label{fig:DegenerateCycles}
\end{figure}

We have all the elements to prove the following:
\begin{teo}\label{teo:amq}
    Consider the system \eqref{eq_modelo1} and let $C=-AMQ$ and $A-M+AM>\frac{1}{Q}$.
    \begin{enumerate}
        \item Let the parameters $(A,M,Q)\in(0,1)\times(-1,0)\times \mathbb R_{>0}$ be chosen such that the fold point $P$ is a generic jump point. In particular, this implies that the equilibrium point $E_1=(\mathcal U_1,\mathcal V_1)$ is on the left branch of $\mathcal M_0^1$ with $0<\mathcal U_1<u_p$. Assume that $E_1$ lies within a distance of order $\mathcal O(1)$ from either singularity $T_\mathcal{C}$ and $P$.  For $\ve=0$, there is a singular orbit $\gamma_0$ as shown in the left panel of Figure \ref{fig:DegenerateCycles}. For $S=\ve\ll1$ there is a locally stable cycle $\gamma_\ve$ that converges in Haussdorf distance to $\gamma_0$ as $\ve\to0$.
        \item Let the parameters $(A,M,Q)\in(0,1)\times(-1,0)\times \mathbb R_{>0}$ be chosen such that the fold point $P$ is a canard point, i.e. $Q=Q_H$, where $Q_{H}$ was given in \eqref{QH}.  For $\ve=0$, there is a singular orbit $\tilde\gamma_0$ as shown in the right panel of Figure \ref{fig:DegenerateCycles}. For $S=\ve\ll1$ and $Q=Q_c(\ve)\approx Q_H$ where $Q_c$ is the value of the parameter $Q$ for which a maximal canard exists, there is a locally stable cycle $\tilde\gamma_\ve$ that converges in Haussdorf distance to $\tilde\gamma_0$ as $\ve\to0$.
    \end{enumerate}
\end{teo}
\begin{proof}
    We refer to suitably defined sections as shown in Figure \ref{fig:DegenerateCycles}. For both cases, the transition $\Sigma_2\to\Sigma_3$ follows from Fenichel's theory, as does $\Sigma_4\to\Sigma_0$. In addition, such transitions are exponential contractions. The transitions $\Sigma_3\to\Sigma_4$, which is also a contraction, and $\Sigma_3\to\Sigma_1$ follow from \cite{Krupa2001ext}. The transitions $\Sigma_0\to\Sigma_2$ and $\Sigma_1\to\Sigma_2$ are both contractions and are the blow-down versions of Proposition \ref{prop:singtr}. Hence, the existence and stability of $\gamma_\ve$ and of $\tilde\gamma_\ve$ follow from the concatenation of the aforementioned maps.
\end{proof}

To finish this section, we present in Figure \ref{fig:sims} a sample of numerical simulations representative of Theorem \ref{teo:amq}. To have a clearer impression of the values of bifurcation parameter $Q$ we chose its values as $Q=Q_H-\delta$, where the corresponding value of $\delta$ appears at the top of each figure.

\begin{figure}[htbp]
    \centering
    \includegraphics[]{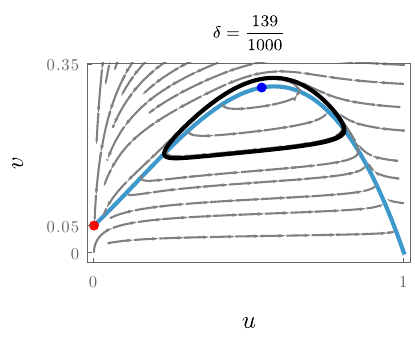} \includegraphics[]{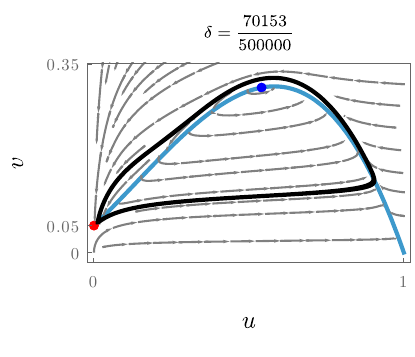} \\
    \includegraphics[]{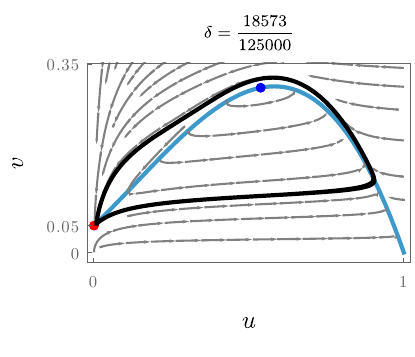}
    \includegraphics[]{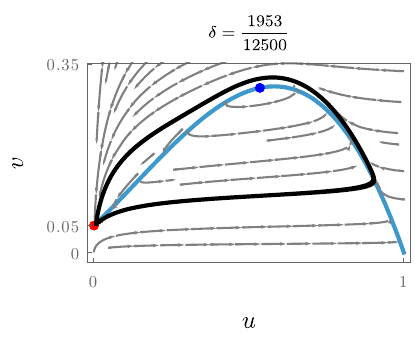} \\
    \includegraphics[]{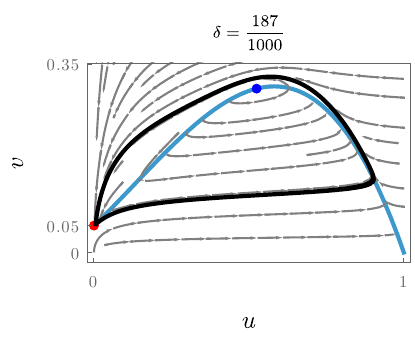}
    \includegraphics[]{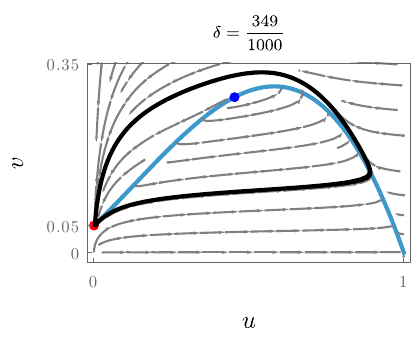}
    \caption{Simulations of \eqref{eq_modelo1} with $S=\ve=0.05$, $A=\frac{1}{2}$, $M=-\frac{1}{10}$, and $C=-AMQ$. The equilibrium point $E_1$ is located close to the fold point $P$\rev{, and is indicated in blue}. The parameter $Q$ is chosen as $Q=Q_H-\delta$ where $\delta$ is shown at the top of each figure. Notice, indeed, the canard explosion occurring for the chosen parameters. We highlight that some of the canard cycles pass through the degenerate point $T_\mathcal{C}$\rev{, indicated in red}.}
    \label{fig:sims}
\end{figure}

\section{Relaxation oscillations for \texorpdfstring{$C<-AMQ$}{C<-AMQ}}\label{sec:relaxationoscillations}

% \hjk{I think we should be clearer between the difference of this and the previous section. In the previous section, due to the degenerate transcritical point, the entry-exit function cannot be used...} 
% \textcolor{blue}{since the divergence integral for $C=-AMQ$ cannot be determined.}  

In Section \ref{sec_desing}, we have studied dynamics organized by a singular Hopf bifurcation and a degenerate transcritical point, $T_C$. In this section, we consider the case where $T_C$ is instead a generic transcritical point. The advantage is that, in this case, one can study the transition across $T_C$ using the enty-exit function \cite{Maesschalck}, which is not possible in the degenerate scenario.
%In the last section, we construct a solution segment of system \eqref{slowfast1} that passes through the neighborhood of the fold point for $\varepsilon \neq 0$. Here, we will employ the entry-exit function, which plays an important role in obtaining the existence of periodic orbits that exhibit relaxation oscillations of a slow-fast system. For more details see  \cite{Maesschalck} and references therein. 
Moreover, we will show the existence of an orbit that jumps from the fold point to the other by attracting the slow manifold  $\mathcal{M}_0^0$  through fast horizontal flow and continuing there for a constant time, the orbit leaves $\mathcal{M}_0^0$ at a certain point. 

To do this, first, we restrict our attention to the system \eqref{eq_modelo1} written as follows
\begin{equation}
\label{Sis44}
\begin{split}
\dfrac{\dd u}{\dd t}&= uf_{1}(u,v)=\: u\left( (u+C)\big((u+A)(1-u)(u-M)-v \big)\right) ,\\
\dfrac{\dd v}{\dd t}&= \varepsilon vg_{1}(u,v)=\:\varepsilon v\left( S(u+A)(u+C-Qv)\right) .
\end{split}
\end{equation}
A straightforward calculation shows us that  $f_{1}(0, v) =C(-AM-v)$, $g_{1}(0, v) = SA(C-Qv)$, which implies that $ f_{1}(0, v) < 0$ if $v > -AM$, $f_{1}(0, v) > 0$ if $v < -AM$. 
${\mathcal T_C}(0, -AM)$ on the vertical axis is the transcritical bifurcation point and we can split the slow manifold $\mathcal{M}_{0}^{0}$ into two parts $\mathcal{ M}_{0}^{0,a}=\lbrace (u,v):u=0,v>-AM\rbrace$ and $\mathcal{M}_{0}^{0,r}= \lbrace (u,v):u=0,v<-AM\rbrace$. Indeed, $\mathcal{M}_{0}^{0,a}$ is attracting and $\mathcal{M}_{0}^{0,r}$ is repelling.

Let us set $\varepsilon = 0$ and let $u_p$ be the maximum point of the critical manifold $\mathcal{M}_0^1$, already given in \eqref{fold_point}, which we recall is
\begin{equation}
\label{eq45}
u_{p}= \dfrac{1}{3} \big(1-A+M+\beta\big),
\end{equation}
where $\beta=\: \sqrt{1+A+A^2-M+AM+M^{2}}$, and from the definition of $\mathcal{M}_{0}^{1}$, it follows that
\begin{equation}
\label{eq46}
v_p=(u_p+A)(1-u_p)(u_p-M).
\end{equation}

Next we  consider a trajectory starting from a point, say $(u_1, v_p)$, where $u_1 < u_p$. 
The trajectory is attracted towards the attracting manifold $\mathcal{M}_{0}^{0,a}$ and starts moving downward maintaining proximity to $\mathcal{M}_{0}^{0,a }$. It is to be expected that the trajectory would leave the vertical axis at the 
bifurcation point ${\mathcal T_C}$ where it loses its stability, see \cite{Muratori1989}. The trajectory crosses the point ${\mathcal T_C}$ and continues moving vertically downwards staying close to the repelling part $\mathcal{M}_{0}^{0,r}$, for a certain time, until it reaches a minimal population of predators $p(v_{p})$ such that $0 < p(v_{p}) < -AM$. After leaving the slow manifold close to the point $p(v_{p})$, the trajectory begins to move along a nearly horizontal segment and is attracted to the slow attractor manifold $\mathcal{M}_{1} ^{a}$. This exit point is defined by  an implicit function $p(v_p)$, called entry-exit function  \cite{Maesschalck},  given by the formula
$$ I(p(v_p)):= \int_{p(v_{p})}^{v_{p}}\frac{f_{1}(0,v)}{vg_{1}(0,v)}\dd v= 0.$$

Let define $v_0 = p(v_p)$, then we have
\begin{equation}\label{eq47}
I(v_0)= \int_{v_{0}}^{v_{p}}\frac{C(-AM-v)}{v (SA(C-Qv))}\dd v =\frac{(AMQ+C)}{AQS} \ln \left( \dfrac{C-Q v_{p}}{C-Q v_{0}}\right) 
+\frac{M}{S} \ln \left( \dfrac{v_{0}}{v_{p}}\right).
\end{equation}

\begin{prop}
For any $v_p$, there exists a unique  $\dfrac{C}{Q} < v_0 < -AM$ such that $I(v_0)=0$.
\end{prop}

\begin{proof}
We start by observing that $AMQ+C<0$. Hence, 
it is obvious that $\displaystyle{\lim_{v_{0}\to \frac{C}{Q}^{+}}I(v_{0})=-\infty}$. 

Since $-\dfrac{AM}{v_{p}}$ and  $\dfrac{C-Qv_{p}}{C+AMQ}$ take values in the interval
$(0,1)$,  one gets:
$$I(-AM)=\frac{(AMQ+C)}{AQS} \ln \left( \dfrac{C-Q v_{p}}{C+QAM}\right) 
+\frac{M}{S} \ln \left( \dfrac{-AM}{v_{p}}\right)>0.$$
So that, the sign change guarantees that there is a zero of $I(v_0)$  somewhere in  $(\frac{C}{Q}, -AM)$.

On the other hand, we have that   $AM+v_0 <0$ and $C-Qv_0<0$, which  imply that
$$\frac{dI(v_{0})}{\dd v_{0}}=\frac{C(AM+v_{0})}{ASv_0(C-Qv_0)}>0.$$
It  follows  that  $I(v_0)$ is a continuous increasing function  on $(\frac{C}{Q}, -AM)$.

From the previous analysis, we conclude that  there is only one  $p(v_0) \in (\frac{C}{Q},-AM)$ such that $I(v_0)=0$.
\end{proof}

Note that if we substitute $v_p$ of \eqref{eq46} into the equation \eqref{eq47}, we get a transcendental equation in $v_{0}$ that we solve numerically to obtain the exit point.

As a consequence, we  have the  next result that shows the existence and uniqueness of the relaxation oscillation.

\begin{teo}
Let ${P}$ be the fold point on the critical manifold $\mathcal{M}_0^1$.
%where the slow flow on the attracting manifold $\mathcal{M}_{0}^{1,a}$  is given by \eqref{m01a}. 
Also, suppose that  the positive equilibrium  equilibrium point is at the normally hyperbolic repelling critical submanifold under
the parametric restriction,
$$\mathcal{U}_{1}<u_p$$
and let  ${\Gamma}$ denote a small neighborhood on a singular trajectory $\gamma_0$
formed by alternating slow and fast trajectories 
$\mathcal O(\ve)$-away from the fold point.
Then, for $0< \varepsilon \ll 1$ there exists a unique
attracting limit cycle $\gamma_{\varepsilon}\subset {\Gamma}$ such that $\gamma_{\varepsilon}\to \gamma_{0}$ (in Haussdorff sense)
as $\varepsilon \to 0$.

%\hjk{I do not think this is rigorous enough. The coexisting equilibrium could be in the normally repelling part and induce a canard. Maybe we have to say $\mathcal O(\ve)$-away from the fold ponit??} 
\end{teo}

\begin{proof}
In order to study the dynamics of the system \eqref{Sis44}, we define two sections of the flow as
\begin{equation*}
\begin{split}
\Delta^{\In}&=\lbrace (u_{+},v):u_{+}\ll u_{p},v \in (v_{p}-\delta,v_{p}+\delta)\rbrace ,\\
\Delta^{\out}&=\lbrace (u_{+},v):u_{+}\ll u_{p},v \in (v_{0}-\delta^{2},v_{0}+\delta^{2})\rbrace ,
\end{split}
\end{equation*}
where $\delta$ is a sufficiently small positive number; see  Figure \ref{fig:inout}.
\begin{figure}[htbp!]
    \centering
    \includegraphics[scale=0.5]{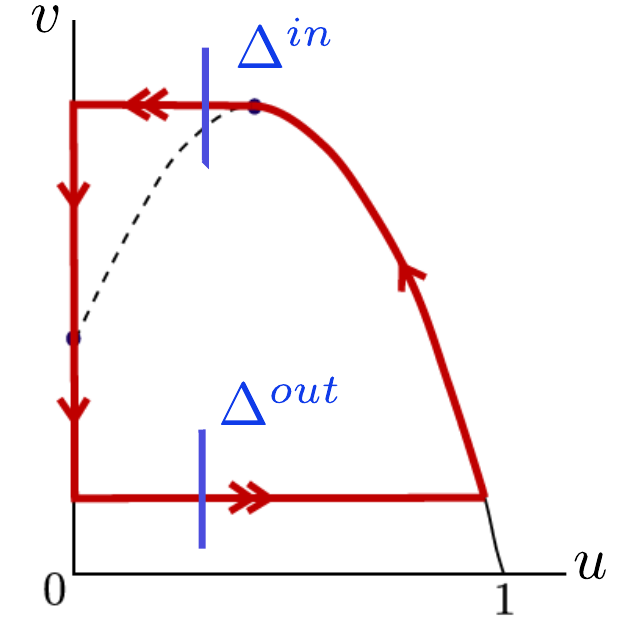}
\caption{Sections of the flow: $\Delta^{\In}$ and $\Delta^{\out}$.}\label{fig:inout}
\end{figure}

We define a flow return map  $\Pi:\Delta^{\In}\rightarrow \Delta^{\In}$, 
by the composition of the following two maps
$$\Phi:\Delta^{\In}\to \Delta^{\out}, \qquad \Psi:\Delta^{\out}\to \Delta^{\In},$$
namely  $\Pi=\Psi \circ \Phi$.
For a fixed $0< \varepsilon \ll 1$, the trajectory of system \eqref{Sis44} 
starting at a point $(u_+, v_+)$ on  section $\Delta^{\In}$. From the analysis of the entry-exit function, 
we are able to say that this trajectory will be attracted to $\mathcal{M}_{0}^{0,a}$ and will leave
 $\mathcal{M}_{0}^{0,r}$  at point  $(0, p(v_{+}))$, where $p$ is the entry-exit function.
 Then, the trajectory  jumps into the section $\Delta^{\out}$ at the point $(u_+, p(v_+))$. 
 Therefore, we can define the map  $\Phi$  by entry-exit function as
 $\Phi(u_+, v_+) = (u_+, p(v_+))$.
 
On the other hand, for the map  $\Psi$ we consider two trajectories $\gamma_{\varepsilon}^{1}$ and  $\gamma_{\varepsilon}^{2}$ starting from the section $\Delta^{\out}$. These trajectories get attracted toward $\mathcal{M}_{\varepsilon}^{1,a}$ 
where the slow flow is given by
$\dfrac{\dd u}{\dd t}= \dfrac{g\big( u,h(u,\varepsilon)\big)}{\dot{h}(u,\varepsilon)}$.
Fenichel's theory and Theorem 2.1 of \cite{Krupa2001ext}  guarantee that $\gamma_{\varepsilon}^{1}$ and  $\gamma_{\varepsilon}^{2}$ contract exponentially toward each other even after jumping into the section $\Delta^{\In}$, so $\Pi$ is, overall, a contraction. 
Finally, the contraction mapping theorem implies that  the map $\Pi$ has a unique fixed point, 
which gives rise to a unique relaxation oscillation cycle $\gamma_\varepsilon$ that converges to the singular 
slow-fast cycle  $\gamma_{0}$ as $\varepsilon \rightarrow 0$. 
%The corresponding bifurcation diagram for fixed $\varepsilon$ is shown in Figure \ref{fig_canard1}\hjk{I am not sure what the relevance of the Figure is for this proof}
\end{proof}

\section{Numerical analysis and validation}\label{sec:numerics}
In this  section,  MatCont is utilized as the principal instrument to examine the qualitative characteristics of the model and 
illustrate the analytical results through parameter variations.

For our simulations, we take the following setup of values of the parameters: $A=\frac{1}{2}$, $M=-\frac{1}{10}$, $\ve=\frac{1}{20}$, while $(C,Q)$ are going to be bifurcation parameters. In Figure \ref{fig:bifurcation-full}, we present a codimension $2$ bifurcation diagram showing the main organization of the important bifurcations. In particular, we can distinguish in Figure \ref{fig:bifurcation-full} regions where the positive equilibrium is unique and where it is not. We further notice that indeed, as shown in Appendix \ref{TB}, there is a Takens-Bogdanov bifurcation close to the point $T_\mathcal{C}$. In there, we verify that the position of the Takens-Bogdanov point obtained numerically is indeed correct. In Figure \ref{fig:bif_LC}, we show some representatives of the structure of the limit cycles in different regions of the parameter space. In particular, we \rev{observe that as soon as there are multiple equilibria, recall Remark \ref{remark2}, the limit cycles corresponding to a canard explosion end up in a homoclinic orbit. The rigorous analysis of \eqref{slowfast1} when multiple equilibria exist is a topic of future research.}
 
\begin{figure}[htbp!]
    \centering
    \begin{tikzpicture}
        \node at (0,0){\includegraphics{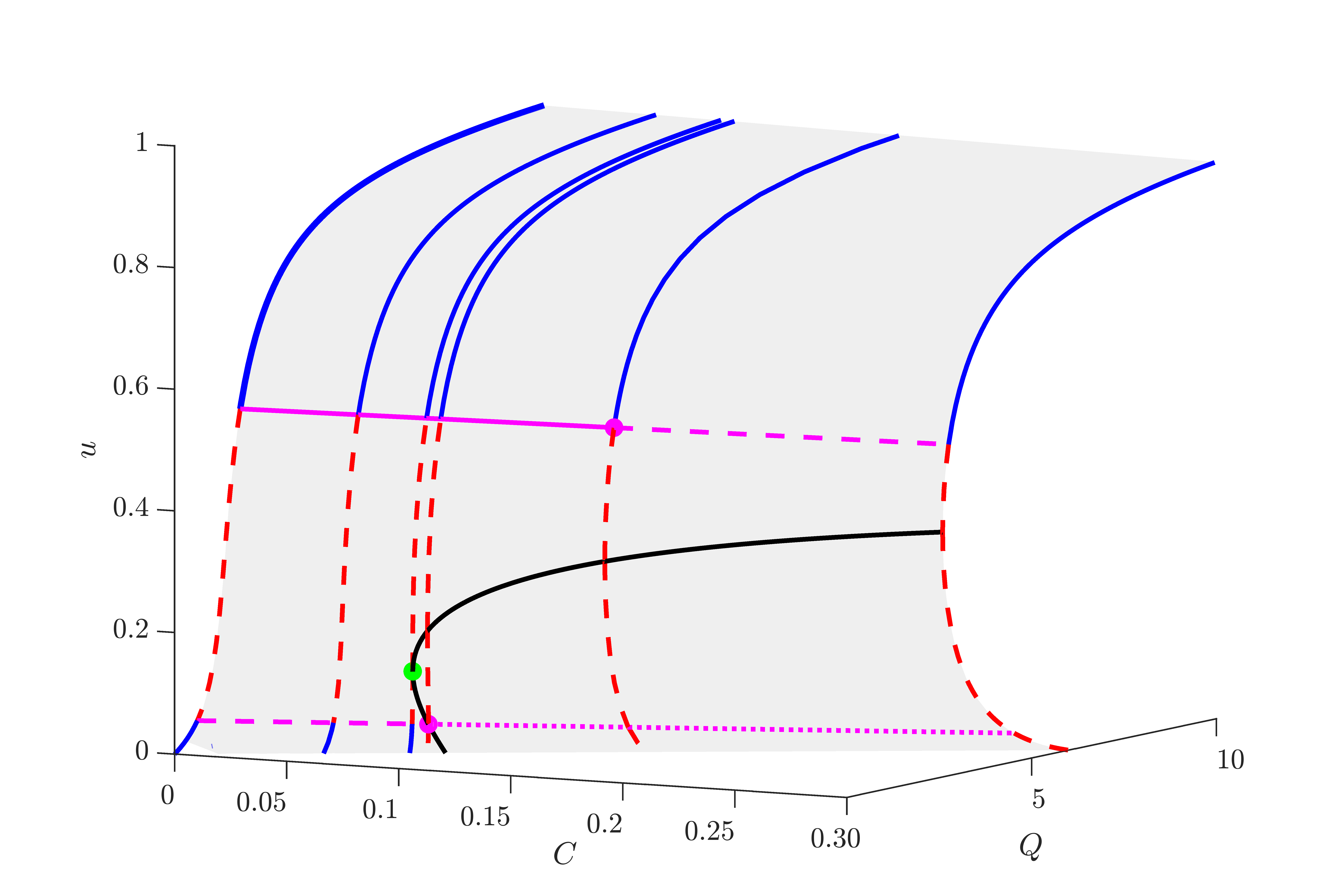} };
        \node[magenta] at (-0.25,0.5) {
        $GH$
        };
        \node[green!80!black] at (-3.35,-2.5) {
        $CP$
        };
        \node[magenta] at (-2.35,-2.9) {
        $BT$
        };
        \node[black] at (1,-0.9) {
        $LP$
        };
    \end{tikzpicture}
    \caption{Equilibria of \eqref{slowfast1} and their bifurcations in a codimension $2$ bifurcation diagram obtained via Matcont \cite{Dhooge2003MATCONTAM}. For this diagram we fix the parameters $A=\frac{1}{2}$, $M=-\frac{1}{10}$ and $\varepsilon=\frac{1}{50}$. The gray surface corresponds to equilibria. For fixed values of $C$, we show stable equilibria in blue and unstable equilibria in dashed red. These equilibria undergo a supercritical Hopf bifurcation along the solid magenta line and a subcritical one along the dashed magenta line. The dotted magenta line indicates neutral saddles. The black curve is a curve of limit points (LP), representing the region where the positive equilibrium is not unique. We further indicate a generalized Hopf point $GH=(u,v;C,Q)\approx(0.544,0.306;0.158,2.292)$, a cusp point $CP=(u,v;C,Q)\approx(0.133,0.128;0.078,1.657)$, and a Takens-Bogdanov point $BT=(u,v;C,Q)\approx(0.047,0.076;0.084,1.721)$. See also Figure \ref{fig:bif_LC}, which complements this diagram. }
    \label{fig:bifurcation-full}
\end{figure}
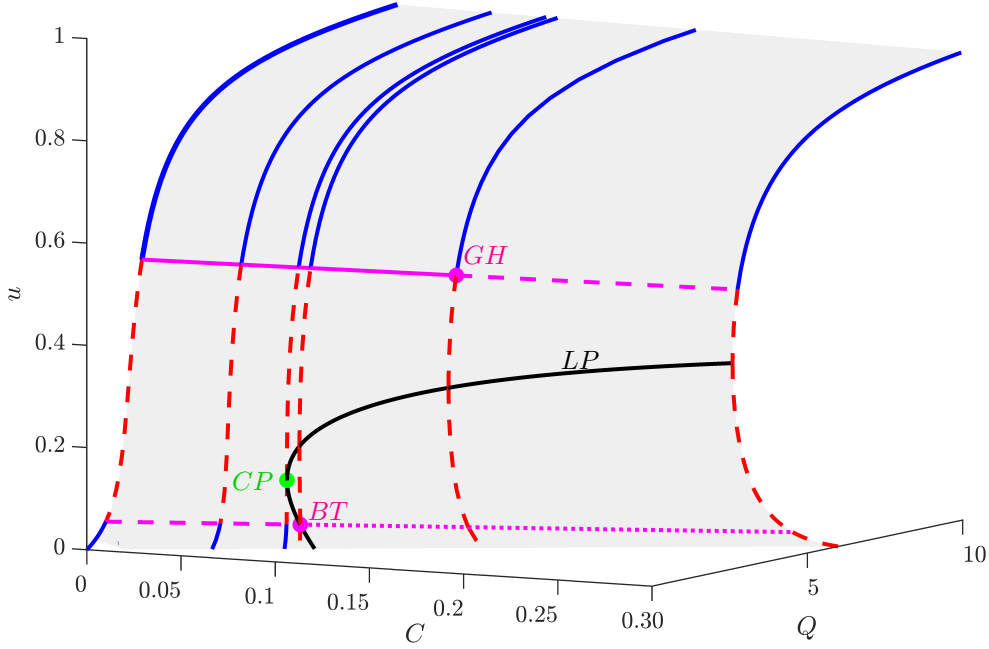

%
% \begin{figure}
%     \centering
%     \includegraphics{figs/bif_diagram_QU.eps}\hfill
%     \includegraphics{figs/bif_diagram_QU_zoom.eps}
%     \caption{Projections of Figure \ref{fig:bifurcation-full} to the $(Q,u)$-plane \hjk{I don't think these are necessary}}
%     \label{fig:enter-label}
% \end{figure}
%
\begin{figure}
    \centering
    \begin{tikzpicture}
        \node at (-3,-6){\includegraphics{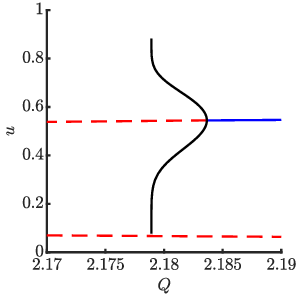}};
        \node at (0,0){\includegraphics{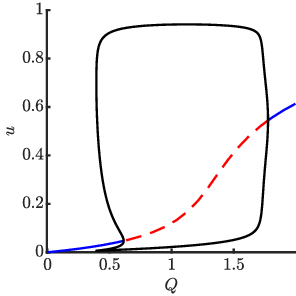}};
        \node at (3,-6){\includegraphics{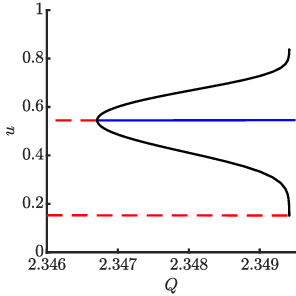}};
        \node at (0,3) {$C=0$};
        \node at (-3,-3.5) {$C=0.125$};
        \node at ( 3,-3.5) {$C=0.175$};
        \node at (1.75,0.5) {$H$};
        \node at (-0.75,-1.4) {$H$};
        \node at (-1.76,-5.3) {$H$};
        \node at ( 1.9,-5.3) {$H$};
    \end{tikzpicture}
    \caption{Bifurcation diagram highlighting the limit cycles (black curves) arising from Hopf bifurcations (compare with Figure \ref{fig:bifurcation-full}). We show corresponding diagrams for $C=0$, $C=0.125$, and $C=0.175$, the latter two around the $GH$ point. Indeed, in the latter two cases, the limit cycles of the canard explosion are truncated and collide with the curve of equilibria, giving rise to a limit homoclinic orbit. This is also a numerical verification that the periodic orbits arising from a canard explosion end in a limit homoclinic orbit whenever the positive equilibrium is not unique, recall Remark \ref{remark2}.}
    \label{fig:bif_LC}
\end{figure}

\section{Conclusions and Discussion}\label{sec:discussion}

% \hjk{To be completed ...}
% Biologically speaking

% Within the context of predator-prey ecological dynamics, this paper  analyzes  the prey-predator system \eqref{eq_modelo1}, focusing on the coexistence of predators under the assumption that prey exhibit fast dynamics. Our findings differ from previously known results. Due to their high reproduction rates, prey populations recover quickly, while predator populations remain nearly constant. With an ample supply of prey, predator populations begin to grow slowly, eventually reaching a maximum level after a relatively long period, completing the slow cycle.

Within the context of predator-prey ecological dynamics, this paper focuses on the coexistence of predators and preys under the assumption that the prey exhibits fast dynamics.
%
% Our work contributes to the existing studies on two key aspects. First, we analyze system \eqref{eq_modelo1} from the perspective of slow–fast dynamics, distinguishing our approach from that of Zhu and Liu \cite{Zhu2022}. In addition to their findings, we uncover new phenomena, including degenerate transcritcal singularities and degenerate relaxation oscillations. Second, although there is extensive research on the dynamics of other slow-fast predator-prey systems, our study offers unique insights within this context.
%
In addition to existing research in the context of slow-fast predator-prey systems such as \cite{Zhu2022}, our main contribution is that we uncover new phenomena that involve the interaction between degenerate transcritcal and Hopf singularities producing transitory canards and relaxation oscillations passing through those points.

It is important to recall that the study presented here focuses on the weak Allee effect, represented by $m<0$ in system \eqref{eq_modelo} (or $M<0$ in  system \eqref{eq_modelo1}), which guarantees the existence of the equilibrium point 
$T_{\mathcal C} = (0, -MC)$, where if in addition $C=-AMQ$, then a degenerate transcritical singularity occurs. We have found that this point may behave in two ways depending on the parameters: on the one hand, it may behave like a stable node, and on the other hand, it may behave like a saddle. In the first case, orbits near $T_C$, which biologically speaking correspond to a low prey population and a predator population of $\approx-MC$, the system may evolve towards the extinction of the prey, or to a cycle of coexistence where the time spent near $T_C$ is considerably large, and this is precisely due to the degenerate nature of the point.  

Several research problem remain open. Here we discuss a few of them. First, we have found (see Sections \ref{sec:numerics} and \ref{TB}) the existence of a Takens-Bogdanov point. This naturally leads to interesting local dynamics and may influence the global ones. From Figure \ref{fig:bifurcation-full} this bifurcation occurs near the region where two positive coexisting equilibria are nearby, and so such a case may deserve extra attention. Another interesting problem, not discussed in this paper, is that of cyclicity \cite{Renato,YaoHuzak}. It would be interesting to consider it in the degenerate scenario.

%Biologically, this suggests that near this point, the prey population has the potential to grow, but at a slower rate than expected due to its low density and the limiting effect of the weak Allee effect. Meanwhile, predator pressure remains minimal because of their low population.  

%Finally, based on the analysis conducted in this paper, we conclude that the weak Allee effect in the slow-fast system \eqref{slowfast1} drives the emergence of relaxation cycles and oscillations, leading to periodic interactions between predator and prey populations, particularly when the prey population grows much faster than the predator population.

\begin{appendix}
\section{Expression for the trace of the Jacobian \texorpdfstring{${\mathcal J}(E_{1})$}{J(E1)}}\label{app:JacobianTrace}

We derive an expression equivalent to $\tfrac{1}{Q}-J_{11}$, for this we will properly combine the equations \eqref{eq_equilibrio1}
$$\mathcal{U}_{1}^3+(A-M-1)\mathcal{U}_{1}^2 + \left(M-A-AM+\dfrac{1}{Q}\right)\mathcal{U}_{1} + AM+\frac{C}{Q}=0$$
and
$$\frac{1}{Q}-J_{11}=\; \frac{1}{Q}-\Big(A-M+AM+2(1-A+M)\mathcal{U}_{1}-3\mathcal{U}_{1}^{2} \Big),$$
we obtain
\begin{equation*}
\begin{split}
\mathcal{U}_{1}\left( \frac{1}{Q}-J_{11}\right) =&\; \left( \frac{1}{Q}-A+M-AM \right)\mathcal{U}_{1}-(1-A+M)\mathcal{U}_{1}^{2}+\mathcal{U}_{1}^{3} -(1-A+M)\mathcal{U}_{1}^{2}+2\mathcal{U}_{1}^{3}\\
=&\; -3\left( AM+\frac{C}{Q}\right)-2\left(\frac{1}{Q}-A+M-AM \right)\mathcal{U}_{1}+(1-A+M)\mathcal{U}_{1}^{2},
\end{split}
\end{equation*}
solving for $wJ_{11}$ we obtain
$$\mathcal{U}_{1}J_{11}=3\left(AM+\frac{C}{Q}\right)+\mathcal{U}_{1}\left((A-M-1)\mathcal{U}_{1}+2\left(\frac{1}{Q}-A+M-AM \right) +\frac{1}{Q} \right). $$
Thus, the  trace  and the determinant of   $\mathcal{J}$ are given by
$$\text{tr}(\mathcal{J})=\: (\mathcal{U}_{1}+C)\Big((A-M+AM)\mathcal{U}_{1}+2(M+1-A)\mathcal{U}_{1}^{2}-3\mathcal{U}_{1}^{3}-S (\mathcal{U}_{1}+A)\Big),
$$
$$\det(\mathcal{J})=\: S \mathcal{U}_{1}(\mathcal{U}_{1}+A)(\mathcal{U}_{1}+C)^{2}\left( \frac{1}{Q}-(A-M+AM)\mathcal{U}_{1}+2(M+1-A)\mathcal{U}_{1}^{2}-3\mathcal{U}_{1}^{3}\right).
$$

%\section{Proof of the Theorem \ref{teo_bifTB}}\label{ap_teo_bifTB}
%\hjk{What does this section contain?}

%\section{Stability of the secondary positive equilibrium points on the critical manifold \texorpdfstring{${\mathcal M}_0^{1,r}$}{M0-1r}}\label{app:secondary equilibria}
%\input{sections/secondayequilibria}

\section{Analytic determination %and analysis
of the Takens-Bogdanov bifurcation}
\label{TB}
Motivated by the numerical findings of Section \ref{sec:numerics} we will  show, in this appendix, that \eqref{eq_modelo1} exhibits a  
Takens-Bogdanov bifurcation of codimension 2, when $E_1$ approaches $T_{\mathcal C}$. To achieve our aim, we shall first verify what are the parameter values for which both the trace and the determinant of the Jacobian matrix $\mathcal{J}$ \eqref{matriz1} evaluated at the point $E_1$ are equal to zero simultaneously.

\begin{prop}\label{prop_det0}
The necessary and sufficient conditions for system  \eqref{eq_modelo1} to have a Takens-Bogdanov bifurcation at $E_{1}= (\mathcal{U}_1,\frac{\mathcal{U}_{1}+C}{Q})$, that is such that ${\rm tr}\,\mathcal{J}(E_1)= \det \mathcal{J}(E_1) =0$, are
 \begin{footnotesize}
 \begin{align}    
\label{eq_cond}
%\begin{array}{l}
A_{*}&=\frac{1}{2 Q \left(\varepsilon+3 (\mathcal{U}_1-1)^2\right)}\Big[\Big(3 C+Q \big(-3 (\varepsilon+2) \mathcal{U}_1+\varepsilon-6 \mathcal{U}_1^3+14 \mathcal{U}_1^2\big)+\mathcal{U}_1+3\Big)^2\nonumber\\
&-4Q \big(\varepsilon+3 (\mathcal{U}_1-1)^2\big)(C (3-6 \mathcal{U}_1)+\mathcal{U}_1 (Q\varepsilon (2 \mathcal{U}_1-1)+Q (\mathcal{U}_1 (3 \mathcal{U}_1-10)+5) \mathcal{U}_1-5\mathcal{U}_1+1))\Big]^{1/2}\nonumber\\
&+3 C-3 Q\varepsilon \mathcal{U}_1+Q\varepsilon-6Q\mathcal{U}_1^3+14Q \mathcal{U}_1^2-6 Q \mathcal{U}_1+\mathcal{U}_1+3,\nonumber\\
\\
M_{*}&=\frac{1}{6Q ( \mathcal{U}_1-1)^2}\Big[\Big(3C+Q \big(-3 (\varepsilon+2) \mathcal{U}_1+\varepsilon-6 \mathcal{U}_1^3+14 \mathcal{U}_1^2\big)+ \mathcal{U}_1+3\Big)^2 \nonumber\\
&-4Q\big(\varepsilon+3 ( \mathcal{U}_1-1)^2\big) (C (3-6  \mathcal{U}_1)+ \mathcal{U}_1 (Q\varepsilon(2  \mathcal{U}_1-1)+Q(\mathcal{U}_1 (3  \mathcal{U}_1-10)+5) \mathcal{U}_1-5 \mathcal{U}_1+1))\Big]^{1/2} \nonumber\\
&-3C-Q\varepsilon  \mathcal{U}_1+Q\varepsilon+6Q \mathcal{U}_1^3-14Q \mathcal{U}_1^2+6Q \mathcal{U}_1- \mathcal{U}_1-3.
\nonumber
\end{align}
 \end{footnotesize}
% for which ${\rm tr}\,\mathcal{J}(E_1)= \det \mathcal{J}(E_1) =0$.
%\textcolor{red}{ $C\in I_{1}$ and $S\in I_{2}$ with $I_{1}$, $I_{2}$ intervals centered on $\frac{1}{20}$.}
\end{prop}
\begin{proof} 
The Jacobian matrix corresponding to equations \eqref{eq_modelo1}  evaluated in the point  $E_{1}=(\mathcal{U}_{1},\frac{\mathcal{U}_{1}+C}{Q})$ is 
$$\mathcal{J}(E_{1})=\: \begin{pmatrix}
\mathcal{U}_{1}(\mathcal{U}_{1}+C)\Big((A-M+AM+2(M+1-A)\mathcal{U}_{1}-3\mathcal{U}_{1}^{2}\Big) & -\mathcal{U}_{1}(\mathcal{U}_{1}+C)\\
\frac{S(\mathcal{U}_{1}+A)(\mathcal{U}_{1}+C)}{Q} & -S (\mathcal{U}_{1}+A)(\mathcal{U}_{1}+C)
\end{pmatrix}$$

We can find that  $\text{tr}\,\mathcal{J}(E_{1})=0$ gives
\begin{equation}\label{eq_m1}
M=\frac{AQ ((\mathcal{U}_{1}-2) \mathcal{U}_{1}-\varepsilon)+3 C+\mathcal{U}_{1} (-Q\varepsilon+Q\mathcal{U}_{1}+3)}{Q (A (2\mathcal{U}_{1}-3)+(\mathcal{U}_{1}-2) \mathcal{U}_{1})}.
\end{equation}

On the other hand, from $\det{\mathcal{J}(E_{1})}=0$ one computes the value of $M$, which is given by 
\begin{equation}\label{eq_m2}
M=\frac{AQ(2\mathcal{U}_{1}-1)+Q \mathcal{U}_{1} (3\mathcal{U}_{1}-2)+1}{Q (A+2\mathcal{U}_{1}-1)}.
\end{equation}

Equating \eqref{eq_m1} and \eqref{eq_m2} and solving for $A$, we obtain the  following expression:
 \begin{footnotesize}
 \begin{align*}    
\label{eq_cond}
%\begin{array}{l}
A_{*}&=\frac{1}{2 Q \left(\varepsilon+3 (\mathcal{U}_{1}-1)^2\right)}\Big[\Big(3 C+Q \big(-3 (\varepsilon+2)\mathcal{U}_{1}+S-6 \mathcal{U}_{1}^3+14 \mathcal{U}_{1}^2\big)+\mathcal{U}_{1}+3\Big)^2\\
&-4Q \big(\varepsilon+3 (\mathcal{U}_{1}-1)^2\big)(C (3-6 \mathcal{U}_{1})+\mathcal{U}_{1} (Q\varepsilon (2 \mathcal{U}_{1}-1)+Q (\mathcal{U}_{1} (3\mathcal{U}_{1}-10)+5) \mathcal{U}_{1}-5\mathcal{U}_{1}+1))\Big]^{1/2}\\
&+3 C-3 Q\varepsilon \mathcal{U}_{1}+Q\varepsilon-6Q\mathcal{U}_{1}^3+14Q\mathcal{U}_{1}^2-6 Q \mathcal{U}_{1}+\mathcal{U}_{1}+3,\\
\end{align*}
 \end{footnotesize}
Now, if we substitute this last value into \eqref{eq_m2}, it follows that
 \begin{footnotesize}
 \begin{align*} 
M_{*}&=\frac{1}{6Q ( \mathcal{U}_1-1)^2}\Big[\Big(3C+Q \big(-3 (\varepsilon+2) \mathcal{U}_1+\varepsilon-6 \mathcal{U}_1^3+14 \mathcal{U}_1^2\big)+ \mathcal{U}_1+3\Big)^2 \nonumber\\
&-4Q\big(\varepsilon+3 ( \mathcal{U}_1-1)^2\big) (C (3-6  \mathcal{U}_1)+ \mathcal{U}_1 (Q\varepsilon(2  \mathcal{U}_1-1)+Q(\mathcal{U}_1 (3  \mathcal{U}_1-10)+5) \mathcal{U}_1-5 \mathcal{U}_1+1))\Big]^{1/2} \nonumber\\
&-3C-Q\varepsilon  \mathcal{U}_1+Q\varepsilon+6Q \mathcal{U}_1^3-14Q \mathcal{U}_1^2+6Q \mathcal{U}_1- \mathcal{U}_1-3.
\nonumber
\end{align*}
\end{footnotesize}
\end{proof}
% \hjk{poner una pequena frase donde se muestran los valores numericos the $A*$ y $M*$ al sustituir valores numericos de parametros en las expresiones anteriores y compararlo con la figure \ref{fig:bifurcation-full}}

To illustrate that Proposition \ref{prop_det0} aligns with the numerical findings presented in Figure \ref{fig:bifurcation-full}, we set the parameter values to those indicated in the caption of Figure \ref{fig:bifurcation-full}, namely $C=0.084$, $Q=1.721$, $\varepsilon=S=0.02$ and $\mathcal{U}_{1}=0.047$. This yields $A_{*}=0.493075\approx \frac{1}{2}$ and $M_{*}=-0.114541\approx -\frac{1}{10}$.
As a consequence, the numerical  data and Proposition  \ref{prop_det0} are in agreement.

\end{appendix}
\section*{Acknowledgements}
Roberto Albarran Garc\'{\i}a
was partially supported by a CONAHCyT Mexico postgraduate fellowship No. 924692.

\bibliographystyle{abbrv}
\bibliography{ref_roberto}

\end{document}